\documentclass[a4paper]{article}

\usepackage{anysize}
\usepackage{multicol}
\usepackage{amsthm}
\usepackage{amsmath}
\usepackage{amssymb}
\usepackage{amsfonts}
\usepackage{amsxtra} 
\usepackage{mathrsfs}
\usepackage{color}
\usepackage{graphicx}
\usepackage[all]{xy}
\usepackage{enumerate}
\usepackage{stmaryrd}
\usepackage{booktabs}
\usepackage{tikz}
\usetikzlibrary{arrows}
\usepackage{mdframed}

\marginsize{4cm}{4cm}{3.5cm}{3.5cm}

\title{Paraconsistent and Paracomplete Zermelo-Fraenkel Set Theory}
\author{Yurii Khomskii\thanks{Amsterdam University College and Hamburg University. },  $\;\;\;$  Hrafn Valt{\'y}r Oddsson\thanks{Ruhr-Universit\"at Bochum.}} 

\newcommand{\p}{\medskip \noindent}

\newcommand{\leqqq}{\leftrightarrow}
\newcommand{\PP}{\mathscr{P}}

\newcommand{\till}{{\upharpoonright}}

\newcommand{\ZFC}{{\rm \sf ZFC}}

\newcommand{\ran}{{\rm ran}}
\newcommand{\dom}{{\rm dom}}

\newcommand{\s}[1]{\medskip \noindent \textbf{#1}}

\newcommand{\M}{\mathcal{M}}
\newcommand{\N}{\mathcal{N}}

\newcommand{\IN}{\mathbb{N}}

\newcommand{\bb}{\mathfrak{b}}

\newtheorem{Thm}{Theorem}[section]

\newtheorem{Lem}[Thm]{Lemma}
\newtheorem{Cor}[Thm]{Corollary}

\theoremstyle{definition}
\newtheorem{Def}[Thm]{Definition}

\newtheorem{Remark}[Thm]{Remark}

\newtheorem{Convention}[Thm]{Convention}

\newcommand{\IZF}{\mathrm{IZF}}

\newcommand{\CZF}{\mathrm{CZF}}
\renewcommand{\ZFC}{\mathrm{ZFC}}
\newcommand{\BZFC}{\mathrm{BZFC}}
\newcommand{\PZFC}{\mathrm{PZFC}}
\newcommand{\ACLA}{\mathrm{ACLA}}
\newcommand{\BS}{\mathrm{BS4}}
\renewcommand{\bb}{\mathfrak{b}}
\newcommand{\nn}{\mathfrak{n}}
\renewcommand{\iff}{\;\;\; \Leftrightarrow \;\;\;}
 \newcommand{\no}{{\sim}} 
 \newcommand{\abbr}{\;\;\; \text{abbreviates} \;\;\; }
\newcommand{\IW}{{\mathbb{W}}}
\newcommand{\HCL}{{\textup{HCL}}}
\newcommand{\IHCL}{{\mathbb{HCL}}}
\newcommand{\Clas}{{\mathrm{C}}}

\begin{document}

\maketitle

\begin{abstract}
 
 We present a novel treatment of set theory in a four-valued {\it paraconsistent} and {\it paracomplete}  logic, i.e., a logic in which propositions can be both true and false, and neither true nor false.  Our approach is a significant departure from previous research in paraconsistent set theory, which has almost exclusively been motivated by a desire to avoid Russell's paradox and fulfil naive comprehension.   Instead,  we prioritise setting up a system with a clear ontology of non-classical sets, which can be used to reason informally about incomplete and inconsistent phenomena,  and is sufficiently similar to $\ZFC$ to enable the development of interesting mathematics.
 
We propose an axiomatic system $\BZFC$, obtained by  analysing the  $\ZFC$-axioms and translating them to a four-valued setting in a careful manner, avoiding many of the obstacles encountered by other attempted formalizations.	We  introduce the {\it anti-classicality axiom} postulating the existence of non-classical sets, and prove a surprising results stating that the existence of a single non-classical set is sufficient to produce any other type of non-classical set.

Our theory is naturally bi-interpretable with $\ZFC$, and provides a philosophically satisfying view in which non-classical sets can be seen as a natural extension of classical ones, in a similar way to the non-well-founded sets of Peter Aczel \cite{AFA}.

Finally, we provide an interesting application concerning Tarski semantics, showing that the classical definition of the satisfaction relation yields a logic precisely reflecting the  non-classicality  in the meta-theory.

\end{abstract}

\s{MSC2020 Classification:} 03E70, 03B53, 03B60

\s{Keywords:} Non-classical set theory; paraconsistent set theory; paraconsistent logic; paraconsistent and paracomplete set theory

\maketitle



\section{Introduction}

 The  Zermelo-Fraenkel axiom system, $\ZFC$, is generally accepted as the foundation of  mathematics. $\ZFC$ is  formalized in classical  logic, in which any statement is  either true or false, and cannot be both at the same time.  Is there a possible interest in considering a set theory in which statements can be neither true nor false (paracomplete) or, more significantly, {\it both true and false} (paraconsistent)?
 
\bigskip The most common motivation for developing such  a theory has been the wish to avoid Russell's paradox and maintain some form of naive comprehension, i.e., the axiom scheme ``$\exists x \forall y (y \in x \leqqq \varphi(y))$'' for every $\varphi$.  In spite of the inconsistency of this scheme, many philosophers have found it more natural and intuitive than $\ZFC$, and have proposed alternative solutions, among others by  adopting a  logic in which contradictory statements can co-exist without trivializing. See, for instance, the work of  Graham Priest \cite{Priest-LP},  Greg Restall \cite{NoteNaive},  Thierry Libert \cite{Libert},   the survey by Roland Hinnion \cite{Hinnion},  and the recent book by Zach Weber \cite{Weber}.

Nevertheless, it is not hard to see,  and has  been known for  a long time (see, e.g., \cite{Geach}),  that paraconsistency or paracompleteness alone are not sufficient to sustain naive set theory, since one can easily appeal to a version of {\it Curry's paradox} instead, and consider $$R = \{x : (x \in x) \to \varphi\}$$ for an arbitrary proposition $\varphi$.  Any logic satisfying   {\it Modus Ponens} and the {\it implication-introduction} rule can derive $\varphi$ from the assumption that $R$ is a set,\footnote{The proof is exactly the same as the proof of Russell's paradox with $\varphi$ replacing $\bot$.} showing that naive comprehension can lead to a trivial theory  without  mentioning the {\it negation} connective at all. The only way to truly avoid paradoxes is to consider a logic that either does not have an implication (such as  the {\it logic of paradox} of \cite{Priest-LP}), or whose ``$\to$''-symbol is so far removed from its common usage that it violates  basic principles of reasoning.  Either way,  the price one has to pay seems too high. 





\bigskip 

Perhaps this pre-occupation with naive comprehension explains why most work in paraconsistent set theory has so far remained  speculative and philosophical in nature, never really `lifting off the ground'. 

But what if we focus our attention on other reasons for studying paracomplete and paraconsistent phenomena in mathematics? After all, there are many   `interesting inconsistencies' in mathematics quite aside of the semantic paradoxes. To name a prominent example,  think of  {\it Kunen's inconsistency}, a result that puts an upper bound on the hierarchy of large cardinal axioms \cite{KunenInconsistency}.  It is conceivable that paraconsistency can shed new light on this or related phenomena.  On a more down-to-earth level,  applications are conceivable in computer science, for example in the study of databases or structures with incomplete or inconsistent information.  Following Belnap \cite{Belnap1, Belnap2}, suppose we want an automated system to derive logical consequences from the information  in  a given database.  Presumably,  this system should not be able to derive any consequence whatsoever merely from the fact that $A(x) \land \lnot A(x)$ holds, which could be due to a wrong entry  in the database.  Some other, more specific applications to mathematical problems are listed in Section \ref{conclusion}.

\bigskip If a solid foundational framework for paracomplete and paraconsistent set theory is to be set up, it should, in our view,   satisfy a number of criteria:

\begin{enumerate}
\item It should be intuitive and philosophically motivated.
\item It should be sufficiently similar to $\ZFC$ to enable the development of interesting mathematics in it. 
\item There should be a tangible `ontology', i.e.,  a working mathematician should be able to  understand and `visualize' paracomplete and paraconsistent sets and how to manipulate them, without a  need to resort to formalism or double-check the axioms.
\item The existence of  paraconsistent and paracomplete sets should be guaranteed by virtue of the axioms,  and these sets should `extend' the von Neumann universe of classical sets.  
\item  It should be possible to construct a model for this theory starting from  $\ZFC$, showing that it is no more problematic than classical set theory.

\end{enumerate}

 This paper provides an exposition of a system which, we claim,  satisfies all of these criteria.  To our knowledge, nothing like this has been done before.
 
 \bigskip
 Our theory is established  in the logic  $\BS$:   a  four-valued logic based on early developments  by Dunn \cite{Dunn}, Belnap \cite{Belnap1, Belnap2} and  \cite{Avron}, later appearing under the name CLoNs in \cite{ExtraBS4}, and the predicate version due to  Omori, Sano and Waragai \cite{Observations, BS4complete} (from where we take its name).   These logics are called {\it logics of formal inconsistency} by Carnieli et al  \cite{Carnielli}.  In Section \ref{logic} we will introduce $\BS$, with a focus on semantics. 

The axiom system we propose is called $\BZFC$.  Intuitively, it can be seen as axiomatising a universe that properly extends the classical universe of sets and includes  {\it incomplete} sets ($u$ such that for some $x$ the statement $x \in u \lor x \notin u$ fails to be true) and  {\it inconsistent} sets ($u$ such that for some $x$, $x\in u \land x \notin u$ is true).  Each non-classical set $u$ can be described by a {\it positive} extension  (the collection of all $x$ such that $x \in u$ is true) and a {\it negative} extension  (the collection of all $x$ such that $x \in u$ is false).\footnote{We will actually talk about the complement of the negative extension and call it the $?$-extension; the reason is explained in Section \ref{intuition}.}

 \bigskip
  The first step of our task is finding an appropriate translation of the $\ZFC$ axioms to the non-classical setting. This is more laborious than might be expected, and previous attempts (e.g., \cite{Carnielli}) may well have stumbled over an insufficiently careful treatment of the $\ZFC$ axioms in this regard.  It would be a mistake to copy them literally; instead, one must think what the axioms were designed to achieve, and fine-tune the formulation to match this design. This is done in Sections \ref{intuition} and \ref{pzfc}.

  This leads us to an {\it intermediate} system, which we call $\PZFC$: an appropriate translation of the old axioms  but  not yet guaranteeing that any non-classical sets do exist. Indeed,  $\PZFC$ together with the statement that all sets are classical, is easily seen to be equivalent to $\ZFC$.  The actual system $\BZFC$ is then obtained by adding to $\PZFC$ an \emph{anti-classicality axiom},  postulating the existence of paraconsistent and paracomplete sets, which is done in Section \ref{bzfc}.  Here we prove an unexpected, yet surprisingly simple result showing that, essentially, all `anti-classicality axioms' are equivalent (Theorem \ref{aclaacla}).

\bigskip
Hinging on this crucial fact,  it is not difficult to construct a   model $\IW$ of $\BZFC$ starting from classical $\ZFC$.  On the other hand,  in   $\BZFC$ itself we can  define an inner model $\IHCL$  of `hereditarily classical' sets and prove (in $\BZFC$) that $\IHCL \models \ZFC$.  In addition,  this mutual interpretation is reversible leading to the fact that  $\BZFC$ is {\it bi-interpretable} with $\ZFC$ (Theorem \ref{bi-int}.)  These results are proved  in Sections \ref{model}, \ref{translation} and \ref{biinterpretability}, preceded by Section \ref{math} in which we build up some needed technology.

\bigskip 
Some readers may feel that the bi-interpretability result limits the significance of our theory, but we would argue for quite the contrary.  To us,   the resulting picture is philosophically very satisfying: if one's position favours the existence of  paracomplete and paraconsistent sets,  one  can view $\BZFC$ as describing the true universe,  an enrichment of the usual  universe of classical sets.  $\ZFC$ is the theory of the inner model $\IHCL$ in which  all of classical mathematics  takes place.  As long as we are only interested in classical mathematics, we can stay within $\IHCL$; but whenever we  encounters phenomena better described by paracomplete or paraconsistent sets,  we can look beyond and take full advantage of $\BZFC$.

  On the other hand,  if one is firmly committed to classical logic and cannot accept true contradictions or the lack of excluded middle, one can still consider $\BZFC$  a useful theory,  namely the theory of the  model $\IW$, using the semantics of $\BS$.  All of paraconsistent and paracomplete set theory as layed out in this paper, can then be understood as formal statements that $\IW$ believes to be true; but this fact is established in classical $\ZFC$.
  
  \bigskip Here,  we would like to draw a parallel to another extension of classical set theory, namely the theory ZFA of {\it non-well-founded sets} of Peter Aczel \cite{AFA}.  Note that ZFA is also bi-interpretable with ZFC,  and  a  similar philosophical freedom is available to the reader as above.  At the same time,  ZFA and the concept of non-well-founded sets has   found many useful applications in mathematics, philosophy and computer science.
  

\bigskip
The final Section \ref{modeltheory}  is devoted to a question of interest for the philosophy of logic, namely the relationship between properties of a formal language and the meta-language in which it is defined.  Starting from $\PZFC$, we show that, if a consequence relation $\models$ is formalised with standard   Tarski semantics,  then the logic which is sound and complete with respect to those semantics satisfies the same  paraconsistency and paracompleteness as in the meta-theory.

\bigskip The work in this paper was carried out in the course of the Master's thesis of the second author  \cite{HrafnThesis}. On occasion, we will refer to the thesis which goes in more depth on some points, and contains more  details which have been left out of this paper for clarity.
 
\section{The Logic $\BS$} \label{logic}



The logic $\BS$ is  due to \cite{BS4complete} with the exception that we take the contradictory constant $\bot$ as primary instead of the classicality operator $\circ$.  
In this section, the meta-theory is classical $\ZFC$.

\subsection{Syntax and Semantics} The main idea behind  $\BS$   is the  separation of truth from falsity, i.e., if $\varphi$ is a sentence and $\M$ a model, then $\varphi$ can be or not be {\it true} in $\M$ and, independently, can be or  not be  {\it false} in $\M$.  This is achieved by considering two interpretation of all predicate symbols  (a ``true'' and a ``false'' interpretation), adapting the inductive definition for the connectives and quantifiers, and thus obtaining two two satisfaction relations: $\models^T$ and $\models^F$.  For convenience we will  consider vocabularies without function symbols.

  
\begin{Def} (The Syntax of $\BS$) 

\p The {\it syntax} of $\BS$ is the usual syntax of first order logic, except that we use  ``$\sim$'' to denote negation.  We also  use the constant connective $\bot$.

\end{Def}

\begin{Def} \label{sem1} (T/F-models)  

\p Suppose $\tau$ is a vocabulary with constant and relation symbols. A {\it T/F-model} $\M$ consist of a {\it domain $M$} together with the following:

\begin{itemize}

\item An element $c^\M \in M$ for every constant symbol $c$.

\item For every $n$-ary relation symbol $R$,  a ``positive'' interpretation  $(R^\M)^+ \subseteq M^n$ and a ``negative'' interpretation $(R^\M)^- \subseteq M^n$.

\item A binary relation $=^+$   which coincides with the true equality relation; and a binary relation $=^-$ satisfying  $a =^- b$ iff $b =^- a$. 

\end{itemize}


\end{Def}

\begin{Def}  \label{mainsemantics} (T/F-semantics for $\BS$)

\p Suppose $\tau$ is a vocabulary without function symbols and $\M$ a T/F-model.  We define  $\models^T$ and $\models^F$ inductively:

{
\begin{enumerate}[1.]

\item $\M \models^T (t=s)[a,b] \iff a=^+ b$.

 $\M \models^F (t=s)[a,b] \iff a =^- b$.

\smallskip\item $\M \models^T R(t_1, \dots, t_n)[a_1 \dots a_n] \iff R^+(a_1, \dots, a_n)$ holds.

$\M \models^F R(t_1, \dots, t_n)[a_1 \dots a_n] \iff R^-(a_1, \dots, a_n)$ holds.

\smallskip \item $\M \models^T  \no \varphi \iff \M \models^F \varphi$.

$\M \models^F \no  \varphi \iff \M \models^T \varphi$.

\smallskip \item $\M \models^T  \varphi \land \psi   \iff \M \models^T \varphi$ and $\M \models^T \psi$.

  $\M \models^F  \varphi \land \psi   \iff \M \models^F \varphi$ or $\M \models^F  \psi$.

\smallskip \item $\M \models^T  \varphi \lor \psi   \iff \M \models^T \varphi$ or $\M \models^T \psi$.

  $\M \models^F  \varphi \lor \psi   \iff \M \models^F \varphi$ and $\M \models^F  \psi$.

\smallskip \item $\M \models^T  \varphi \to \psi   \iff$ if $ \M \models^T \varphi$ then $\M \models^T \psi$.

  $\M \models^F  \varphi \to \psi   \iff \M \models^T \varphi$ and $\M \models^F  \psi$.
 
\smallskip \item $\M \models^T  \varphi \leqqq \psi   \iff$ ($\M \models^T \varphi$ if and only if $\M \models^T \psi$).

 $\M \models^F  \varphi \leqqq \psi   \iff$ ($\M \models^T \varphi$ and $\M \models^F \psi$) or ($\M \models^F \varphi$ and $\M \models^T \psi$).

\smallskip \item $\M \models^T  \exists x \varphi(x)    \iff \M \models^T  \varphi[a]$ for some $a \in M$.

$\M \models^F  \exists x \varphi(x)    \iff \M \models^F  \varphi[a]$ for all $a \in M$.

 \smallskip \item  $\M \models^T  \forall x \varphi(x)    \iff \M \models^T  \varphi[a]$ for all $a \in M$.

$\M \models^F  \forall x \varphi(x)    \iff \M \models^F  \varphi[a]$ for some $a \in M$.

\smallskip \item  $\M \models^T \bot \iff$ never.

 $\M \models^F \bot \iff $ always.

\end{enumerate}
}
\noindent  If $\M \models^T \varphi$ then we say that $\varphi$ is {\it true} in $\M$, and if $\M \models^F \varphi$ then we say that $\varphi$ is {\it false} in $\M$. 
\end{Def}

 \begin{Def} \label{sem} (Semantic consequence) If $\Sigma$ is a set of formulas and $\varphi$ another formula, then {\it semantic consequence} is defined by \begin{center} $\Sigma \vdash_\BS \varphi$ 

\medskip
iff for every T/F-model $\M$, if $\M \models^T \Sigma$ then $\M \models^T \varphi$.
\end{center}
\end{Def}  \begin{Remark} \label{remark} While most of the inductive steps in Definition \ref{mainsemantics} are straightforward,  two points  need to be addressed:
 
 \begin{enumerate}
 
 \item The  falsum symbol ``$\bot$''  should not be understood  as just   a `contradiction' but rather as a strong form of falsity, one which cannot be satisfied even in `paraconsistent' models.  Some readers might initially find the  addition to   $\bot$ to the logic distasteful, as it seems to run counter to the idea of paraconsistency. However,  in a vocabulary with finitely many relation symbols, one can write the following sentence: 
 $$\forall x\forall y (x=y\land x\neq y)\land \bigwedge_{R  \in S} \forall x_1 \dots \forall x_n (R(x_1,...,x_n)\land \no R(x_1,...,x_n)).$$ 

This sentence {\it is} satisfiable, but only in the trivial model consisting of exactly one object $a$, which is both equal and not equal to itself ($a =^+ a$ and $a =^- a$) and for which all relation-interpretations are {\it true} and {\it false}.  Adding  ``$\bot$'' to the language is equivalent to disregarding this trivial model.  Since we focus on set theory, we will have no interest in such a model and thus have no reservations about $\bot$.
 
 \item There is some freedom in choosing the truth- and falsity-conditions for  the implication. For example, in the  logic LP from \cite{Priest-LP},  an implication $\varphi \to \psi$ would be an abbreviation for $\no \varphi \lor \psi$.  But such a logic fails to satisfy implication-introduction rule and the deduction theorem,  leading to many undesirable consequences, such as finite models of set theory, see \cite{NoteNaive}.  In $\BS$, the truth-definition  for the material implication reflects semantic consequence and makes sure that the deduction theorem is satisfied, while the falsity-condition reflects the existence of a counterexample.  
 
Combining the material implication with $\bot$,  we will be able to  define {\it classical negation}, and  refer not only to  {\it truth} and {\it falsity}, but also to the {\it absence} of truth and/or falsity,  from within the system.  But this should not be seen as a drawback of the system; indeed  the original formulation of $\BS$ from \cite{BS4complete} explicitly contained a  `classicality' operator which generates an equivalent logic.
 
 \end{enumerate}
 

 \end{Remark}

 \begin{Def} (Truth value) For every $\varphi$ and T/F-model $\M$ we define:
   
   $$   
    \llbracket   \varphi\rrbracket^\mathcal{M}  \;\; := \;\; \begin{cases}
    
    1 & \text{if } \M \models^T \varphi \text{ and }\M \not\models^F \varphi \\
     \bb & \text{if } \M \models^T \varphi \text{ and }\M \models^F \varphi  \\
      \nn & \text{if } \M  \not\models^T\varphi \text{ and }\M \not\models^F \varphi  \\
       0 & \text{if } \M \not\models^T \varphi \text{ and }\M \models^F \varphi  \\
       
       \end{cases}  $$
 
 \end{Def}
 \p We can now view $\BS$  as  a four-valued logic with  truth tables for propositional connectives given  in Table \ref{ttables}.

 \begin{table}[h] 

\rule{12cm}{0.02cm}
\begin{center} \footnotesize
\begin{tabular}{c|c}
&${\sim}$\\
\midrule
1&0\\
$\mathfrak{b}$&$\mathfrak{b}$\\
$\mathfrak{n}$&$\mathfrak{n}$\\
0&1\\
\end{tabular}
\hspace{1cm} \vspace{-0.5cm}
\begin{tabular}{c|cccc}
$\land$&1&$\mathfrak{b}$&$\mathfrak{n}$&0\\
\midrule
1&1&$\mathfrak{b}$&$\mathfrak{n}$&0\\
$\mathfrak{b}$&$\mathfrak{b}$&$\mathfrak{b}$&0&0\\
$\mathfrak{n}$&$\mathfrak{n}$&0&$\mathfrak{n}$&0\\
0&0&0&0&0\\
\end{tabular}
\hspace{1cm}
\begin{tabular}{c|cccc}
$\lor$&1&$\mathfrak{b}$&$\mathfrak{n}$&0\\
\midrule
1&1&1&1&1\\
$\mathfrak{b}$&1&$\mathfrak{b}$&1&$\mathfrak{b}$\\
$\mathfrak{n}$&1&1&$\mathfrak{n}$&$\mathfrak{n}$\\
0&1&$\mathfrak{b}$&$\mathfrak{n}$&0\\
\end{tabular}
  \end{center}
\footnotesize
\begin{center}
\vspace{0.5cm}
\begin{tabular}{c|cccc}
$\rightarrow$&1&$\mathfrak{b}$&$\mathfrak{n}$&0\\
\midrule
1&1&$\mathfrak{b}$&$\mathfrak{n}$&0\\
$\mathfrak{b}$&1&$\mathfrak{b}$&$\mathfrak{n}$&0\\
$\mathfrak{n}$&1&1&1&1\\
0&1&1&1&1 
\end{tabular}
\hspace{1cm}
\begin{tabular}{c|cccc}
$\leftrightarrow$&1&$\mathfrak{b}$&$\mathfrak{n}$&0\\
\midrule
1&1&$\mathfrak{b}$&$\mathfrak{n}$&0\\
$\mathfrak{b}$&$\mathfrak{b}$&$\mathfrak{b}$&$\mathfrak{n}$&0\\
$\mathfrak{n}$&$\mathfrak{n}$&$\mathfrak{n}$&1&1\\
0&0&0&1&1\\
\end{tabular}
  \end{center}

\rule{12cm}{0.02cm}
  
  \caption{Truth tables for the propositional connectives of $\BS$} 

\label{ttables} \end{table}

\subsection{Defined connectives} \label{defined}

We will need several defined connectives to make the presentation more smooth and intuitive. 

 First let us consider material implication: notice that $\M \models^T  \varphi \to \psi$ tells us that if $\varphi$ is true in $\M$ then $\psi$ is true in $\M$, but does not tell us that if $\psi$ is false in $\M$ then $\varphi$ is false in $\M$, as can easily be verified. Similarly, a bi-implication $\varphi \leqqq \psi$ tells us that, in a model $\M$,  $\varphi$ is true iff $\psi$ is true, but not that $\varphi$ is false iff $\psi$ is false. In particular, $\varphi \leqqq \psi$ does not allow us (as we are used from classical logic) to substitute an arbitrary occurrences of $\varphi$   with $\psi$ within a larger formula. 
 
 For this reason, we define the following abbreviations, which we call {\it strong implication} and {\it strong bi-implication}, respectively.\footnote{The strong implication appears, e.g., in   \cite[Chapter XII]{Strongimplication}.}

 \begin{itemize} \item $\varphi\Rightarrow\psi \;\;\; \text{abbreviates} \;\;\; ( \varphi\rightarrow \psi) \land ( {\sim}\psi\rightarrow{\sim}\varphi) $
 
 \item $\varphi\Leftrightarrow\psi \;\;\;   \text{abbreviates} \;\;\;  (\varphi  \leqqq \psi) \land ( \no \varphi \leqqq \no \psi).$ \end{itemize}

In particular,  if $\varphi \Leftrightarrow \psi$ is true then any occurrence of $\varphi$ may be substituted with $\psi$, and vice versa. The distinction between regular and strong implication and bi-implication will play a crucial role in the correct formulation of the axioms. 

\bigskip Next,  following up on Remark \ref{remark}  we define  {\it classical negation} as follows: 

 \begin{itemize} \item $\neg \varphi \abbr \varphi\rightarrow \bot$ 
\end{itemize}

One can easily check that $\M \models^T \lnot \varphi$ iff $\M \not\models^T \varphi$ while $\M \models^F \lnot \varphi$ iff $\M \models^T \varphi$.  So in a model $\M$, $\lnot \varphi$ can be either true and not false (when $\M \models^T \varphi$) or false and not true (when $\M \not\models^T \varphi$). Using classical negation as a defined notion we can talk about presence and absence of truth and falsity (and, more generally, specify the truth value of a formula) from within the system. We will use the following important abbreviations:  

 \begin{itemize} \item $! \varphi \abbr \no \lnot \varphi $
\item $? \varphi \abbr \lnot \no  \varphi $
\end{itemize}
We think of $!\varphi$ as  {\it presence of   truth} and $?\varphi$ as  {\it absence of  falsity}.  The truth tables for $\no, \lnot, !$ and $?$   in Table \ref{ttnot} make this clear.
Notice that $\lnot \varphi$, $!\varphi$ and $?\varphi$ will always have truth value $1$ or $0$. Moreover, the truth value of $\lnot \varphi$ and $!\varphi$  depends {\it only} on whether $\varphi$ was {\it true} in the model, and completely disregards whether $\varphi$ was {\it false}. Similarly,   $?\varphi$ depends only on whether $\varphi$ was {\it false} and disregards whether it was {\it true}.  

\begin{table}[h]
\bigskip

\begin{center} 
\begin{tabular}{c|ccccc}\toprule
$\varphi$& $\no \varphi$ & $\lnot \varphi$&$ !\varphi $& $?\varphi$\\
\midrule
1 & 0 & 0 & 1 & 1 \\
$\bb$  & $\bb$& 0 & 1 & 0 \\
$\nn$ & $\nn$ & 1 & 0 & 1 \\
$0$ & 1& 1 & 0 & 0 \\
\bottomrule
\end{tabular}
\end{center}


\caption{Truth table for $\no,  \lnot, !$ and $?$. 
} \label{ttnot}
\end{table}

 A $\BS$-formula is  \emph{complete} if it can never obtain the truth-value $\nn$,  and {\it consistent} if it can never obtain the truth-value  $\bb$.  It is called \emph{classical} if it is both complete and consistent, i.e., if in every model it has truth-value $1$ or $0$.  In particular, $\lnot \varphi$, $!\varphi$ and $?\varphi$ are classical formulas for any $\varphi$.  Notice also that classicality, completeness and consistency can each be expressed within the system, by $ {!}\varphi\leqqq {?}\varphi$,  ${?}\varphi\to {!}\varphi$ and $!\varphi \to ?\varphi$, respectively.  We will use the following abbreviation in accordance to  \cite{Avron}: 
 
 \begin{itemize} \item $\circ \varphi \abbr {!}\varphi\leqqq {?}\varphi$ \end{itemize}
 It is easy to see that,  taking   $\circ$   as primary rather than $\bot$, we obtain the same logic. 

\subsection{Proof system} \label{proof}

A sound and complete proof calculus for  $\BS$ is  presented in \cite{BS4complete}.  We use a slightly modified but equivalent version,  with the following axioms and rules of inference:

\begin{itemize}
\item Axioms of classical predicate logic:
    \begin{multicols}{2}
    \begin{enumerate}
	\item
    $\varphi \rightarrow (\psi \rightarrow \varphi)$
    \item
    $(\varphi \rightarrow (\psi \rightarrow \chi)) \rightarrow \\ ((\varphi \rightarrow \psi) \rightarrow (\varphi \rightarrow \chi))$
    \item
    $\varphi\lor (\varphi\rightarrow \psi)$
    \item
    $(\varphi \wedge \psi) \rightarrow \varphi$
    \item
    $(\varphi \wedge \psi) \rightarrow \psi$
    \item
    $\varphi \rightarrow (\psi \rightarrow (\varphi \wedge \psi))$
    \item
    $\varphi \rightarrow (\varphi \vee \psi)$
    \item
    $\psi \rightarrow (\varphi \vee \psi)$
	\item
    $(\varphi \rightarrow \chi) \rightarrow ((\psi \rightarrow \chi) \rightarrow \\ ((\varphi \vee\psi) \rightarrow \chi))$
	\item
    $\bot\rightarrow \varphi$
	\item
    $\forall x\varphi(x)\rightarrow \varphi(t)$
	\item
    $\varphi(t)\rightarrow \exists x\varphi(x)$
    \item
    $x=x$
    \item
    $x=y\rightarrow ( \varphi(x)\rightarrow \varphi (y) ) $
    \end{enumerate}
    \end{multicols}
\item Axioms for negation:
    \begin{multicols}{2}
    \begin{enumerate}\setcounter{enumi}{14}
	\item
    ${\sim} {\sim} \varphi \leftrightarrow \varphi$
	\item
    ${\sim} (\varphi \wedge \psi) \leftrightarrow ({\sim} \varphi  \vee {\sim} \psi)$
	\item
    ${\sim} (\varphi \vee \psi) \leftrightarrow ({\sim} \varphi  \wedge {\sim} \psi)$
	\item
    ${\sim} (\varphi \rightarrow \psi) \leftrightarrow (\varphi  \wedge {\sim} \psi)$
    \item
    ${\sim}\bot$
    \item
    ${\sim}\forall x\varphi\leftrightarrow \exists x {\sim}\varphi$
    \item
    ${\sim}\exists x\varphi\leftrightarrow \forall x {\sim}\varphi$
    \item
    ${\sim}(x=y)\rightarrow{\sim}(y=x)$.\footnote{Axiom 22 does not occur in the original formulation in \cite{BS4complete},  but we need to add it to take care of the semantic requirement that $=^-$ is a symmetric relation.}
    \end{enumerate}
    \end{multicols}
\item The rules of inference:
    \begin{enumerate}\setcounter{enumi}{22}
    \item From $\varphi$ and $\varphi\rightarrow \psi$, infer $\psi$ (modus ponens).
    \item Infer $\varphi\rightarrow\forall x\psi $ from $\varphi\rightarrow\psi$, provided  $x$ does not occur free in $\varphi.$
    \item Infer $\exists x\varphi\rightarrow\psi $ from $\varphi\rightarrow\psi$, provided  $x$ does not occur free in $\psi.$
    \end{enumerate}
\end{itemize}

\begin{Lem} \label{completeness} The calculus described above is sound and complete with respect to T/F-semantics. \end{Lem}

\begin{proof}  This follows by adapting the  proof of  \cite[Corollary 5.15]{BS4complete} to refer to $\bot$ rather than the classicality operator $\circ$  as the primary symbol. We leave out the details. \end{proof}

\medskip In practice,  we will reason informally within the system $\BS$ using arguments formalizable in the calculus.  
We specifically mention some provable statements  concerning defined connectives, which will frequently be needed in  later arguments.

\begin{Lem}  The following statements are provable in $\BS$:\end{Lem}

\begin{itemize}



\item $\varphi \leqqq \;  !\varphi$
\item $\no \varphi \leqqq \no ?\varphi$

 ($!$ talks only about {\it truth} and $?$ only about {\it faslity}.)

   \item $x=y\rightarrow ( \varphi(x)\Leftrightarrow \varphi(y))$

  (A true equality allows us to interchange terms).


   

\item ${\circ \varphi}\;\; \to\;\; ((\varphi \Leftrightarrow \;  !\varphi) \land (\varphi \Leftrightarrow  \; ?\varphi) \land (\no \varphi \Leftrightarrow \lnot \varphi ))$
\item ${\circ \varphi \land \circ \psi} \;\; \to\;\;( (\varphi \to \psi) \Leftrightarrow (\varphi \Rightarrow \psi))$
 
  (For {\it classical} formulas there is no distinctions between strong and weak implication, nor between native and classical negation, and $!$ and $?$ may be omitted).
 \end{itemize}

\section{Non-classical sets} \label{intuition}

Before delving into the axioms,  it is helpful to think about the concept of a set in a paraconsistent and paracomplete setting from a naive point of view. 
In classical ZFC, a {\it set} $x$ is identified with the class of its elements $ \{y : y \in x\}$ and divides the entire universe in two parts: those $y$ that are {\it in}  $x$, and those $y$ that are {\it not in} $x$.  In the context of paraconsistent and paracomplete logic, we can have a situation where a set $y$ is both in $x$ and not  in $x$, or a situation where $y$ is neither in $x$, nor is it the case that $y$ is not in $x$. 

\bigskip Therefore, it seems natural to  identify a set $x$ with the  pair consisting of a {\it positive} extension (those $y$ for which ``$y \in x$'' is {\it true}) and a {\it negative} extension (those $y$ for which ``$y\in x$'' is false),  without the added requirement that one be the complement of the other.   In fact, it makes sense to call a set {\it consistent} if its positive and negative extensions do not intersect, {\it complete} if their union is the whole universe, and  {\it classical} if it is both consistent and complete.

There is, however, an asymmetry here: the positive extension is a set, while the negative extension is a proper class.  Therefore,  it turns out to be more appropriate to talk about the {\it complement} of the negative extension, i.e., those $y$  for which the statement ``$y \in x$'' is {\it not  false} (or,  equivalently for which the statement  ``$y \notin x$'' is {\it not true}). We will  refer to this collection as the  ?-{\it extension} of $x$. 

Although we have just referred to  statements being {\it true} or {\it false}, which are seemingly  meta-theoretic notions, recall that the operators $!$ and $?$ allow us to discuss truth and falsity from {\it within} the system as well.  In particular,   $!(y\in x)$  is {\it true} if and only if $y\in x$ is {\it true},  and $?(y \in x)$  is {\it true} if and only if $y \in x$ is {\it not false}.    This motivates the following:

\begin{Def} Let $x$ be a set: \begin{itemize}

\item The !-{\it extension} of $x$ is
$$x^! := \{y  \;: \; \; !(y \in x)\}$$
\item The ?-{\it extension} of $x$  is 
$$x^? := \{y \; :\;  \; ?(y \in x)\}$$
\end{itemize}
\end{Def}

\p
For the time being,  it is not clear that the above collections describe sets and not proper classes, but we shall set up the axiomatic framework in such a way that if $x$ is a set, then both $x^!$ and $x^?$ are sets. 
The four boolean combinations of  $x^!$ and $x^?$ determine the  classes consisting of all $y$ for which the statement ``$y \in x$'' has one of the four possible truth-values, as  visualized in Figure \ref{balls}. 

\bigskip We remark that the property of sets being complete, consistent and classical can  be expressed within the system. In fact the following holds:



\begin{itemize}

\item $x$ is {\it complete} iff $\forall y \; (y \in x^? \to y \in x^!)$

\item $x$ is {\it consistent} iff $\forall y \; (y \in x^! \to y \in x^?)$

\item $x$ is {\it classical} iff $\forall y \; (y \in x^! \leqqq y \in x^?)$




\end{itemize}

Classical sets behave as we are used to in $\ZFC$, i.e.,  the $!$-extension is exactly the $?$-extension,  and consequently it does not matter whether we consider the native negation $\sim$ or classical negation $\neg$, nor whether we use $\to$ or $\Rightarrow$.  It is not hard to see that $x^!$ and $x^?$ are themselves classical.

   
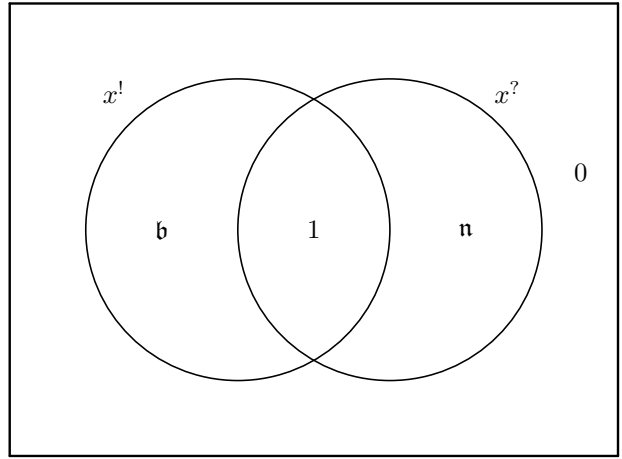
\begin{figure}[h]
\centering
\begin{tikzpicture}[line cap=round,line join=round,>=triangle 45,x=1cm,y=1cm]
\clip(-4.7,-3.3) rectangle (4.4,3.6);
\draw [line width=0.6pt] (-1,0) circle (2cm);
\draw [line width=0.6pt] (1,0) circle (2cm);
\draw [line width=0.9pt] (-4,3)-- (4,3);
\draw [line width=0.9pt] (4,3)-- (4,-3);
\draw [line width=0.9pt] (4,-3)-- (-4,-3);
\draw [line width=0.9pt] (-4,-3)-- (-4,3);
\draw (-2.9,2.1) node[anchor=north west] {$x^!$};
\draw (2.25,2.1) node[anchor=north west] {$x^?$};
\draw (-2,0) node {$\mathfrak{b}$};
\draw (0,0) node {1};
\draw (2,0) node {$\mathfrak{n}$};
\draw (3.3,1) node[anchor=north west] {0};
\end{tikzpicture}
\caption{The  four truth values of ``$y \in x$'' depending on the boolean combination of $x^!$ and $x^?$. } 
\label{balls}
\end{figure}

Let us now look more closely at the notion of {\it equality} of two sets and the extensionality principle.  
In ZFC,   two sets are equal precisely if the classes of their elements are the same; equivalently, two sets are different precisely if one set contains an element which the other does not, or vice versa.

In our setting,   $x=y$ and $x \neq y$\footnote{Again, we write $x\neq y$ instead of $\no (x=y)$.} could both be true statements or it could happen that neither $x=y$ nor $x \neq y$ are true statements.  But we still require that  ``$x=y$''  express the idea  that $x$ and $y$ are  the same set-theoretic objects, meaning that both its !-extension and ?-extension must be  the same. Likewise,  we  still want ``$x \neq y$'' to express that there is something in $x$ which is not in $y$, or vice versa.  So the guiding principle behind the extensionality axiom must be the following: 

\vbox{

\begin{itemize}
\item $x = y \;\; \leqqq \;\;  (x^! =  y^! \land x^? = y^?)$

\item $x \neq y \;\; \leqqq \;\;   \exists z ((z \in x \land z\notin y) \lor (z \in y \land z\notin x))$
\end{itemize}

}

Figure \ref{figgy} illustrates the situation in which two sets $x$ and $y$ are equal (because their !-extensions and ?-extensions coincide) but also unequal (because there is a set $z$ such that $z \in x$ but $z \notin y$).  In fact,   if $x$ is any inconsistent set, then $x=x$ and $x\neq x$.

\begin{figure}[h]
\centering
\begin{tikzpicture}[line cap=round,line join=round,>=triangle 45,x=1cm,y=1cm]
\clip(-4.7,-3.3) rectangle (4.4,3.6);
\draw [line width=0.6pt] (-1,0) circle (2cm);
\draw [line width=0.6pt] (1,0) circle (2cm);
\draw [line width=0.9pt] (-4,3)-- (4,3);
\draw [line width=0.9pt] (4,3)-- (4,-3);
\draw [line width=0.9pt] (4,-3)-- (-4,-3);
\draw [line width=0.9pt] (-4,-3)-- (-4,3);
\draw (-2.9,2.399204957738897) node[anchor=north west] {$x^! = y^!$};
\draw (2.0114062385585285,2.432706796198534) node[anchor=north west] {$x^?  = y^?$};
\draw (-2.1,0) node {$\bullet$};
\draw (-1.5,-0.4) node {$z \in x$};
\draw (-1.5,-0.9) node {$z \notin y$};

\end{tikzpicture}
\caption{$x=y$ and $x \neq y$.} 
\label{figgy}

\end{figure}
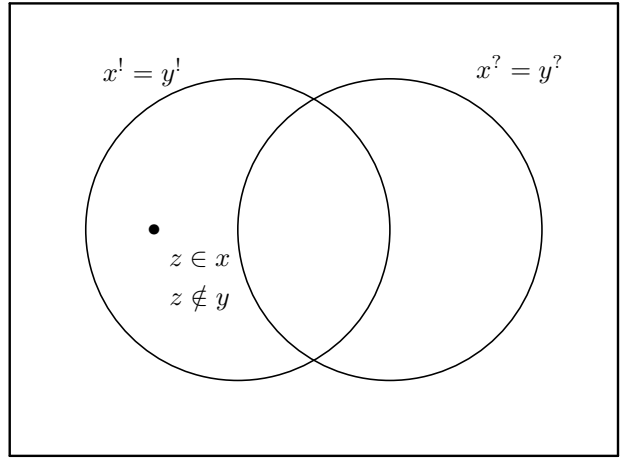

\bigskip It is customary in set theory to use notation such as $\{y : \varphi(y)\}$ to refer to the set (if it exists) of all objects $y$ satisfying property $\varphi$.  In our setting, where a set is determined by its !-extension and ?-extension, it makes sense to agree on the following:

\begin{Convention} \label{conny}

$$x = \{y : \varphi(y)\} \;\;\;\; \text{  abbreviates } \;\;\;\; \forall y \; (y \in x \; \Leftrightarrow \; \varphi(y)).$$

\end{Convention}

\p The  use of the strong implication means that $x^!$ is the set of all $y$ for which $\varphi(y)$ is {\it true}, and $x^?$ is the set of all $y$ for which $ \varphi(y)$ is {\it not false}.

Now it should be clear why the system under consideration cannot avoid Russell's paradox, since $R := \{x :  \lnot (x \in x)\}$ cannot be a set by the usual argument.    Likewise, the universe of all sets does not form a set.

\bigskip A last point of subtlety should be discussed: how do we understand notation such as $\{u\}$, for a set $u$? One might initially assume that this is  $\{y : y = u\}$ (referring to Convention \ref{conny}). However,  a closer look reveals the following: if  $\{y : y = u\}$  is  a set, then we would like its $?$-extension to also be a set.  However,  a bit of work following the definitions shows that this $?$-extension would be the collection of all sets $a$ such that  $a^! \subseteq u^?$ and $u^! \subseteq a^?$.   Since there is no upper bound on $a^?$, it would seem (following classical intuition) that there are class-many  sets $a$ satisfying this condition, meaning that the ?-extension of this set is, in fact, a proper class.\footnote{We will provide a proper proof of this result in Lemma \ref{properclass}} 

Instead, we  opt for the following:

\begin{Convention} \label{conny2} If $u$ is a set, then $\{u\}$ is the set $\{y \: : \: !(y=u)\}$; and more generally if $u_0, \dots, u_n$ are sets, then   $\{u_0, \dots, u_n\}$ is the set $ \{y \:: \: \: !(y=u_0) \lor \dots \lor !(y=u_n)\}$.  This is always a {\it classical} set whose !-extension and ?-extension are exactly the finite set consisting of the elements in question. %

\end{Convention}
A similar phenomenon occurs with other set-theoretic operations, most notably the power set operation.

\section{The theory $\PZFC$} \label{pzfc}

We  now introduce the intermediate system $\PZFC$ by carefully analysing the standard axioms of set theory and generalizing them in  accordance with the intuition described in the previous section. 

\subsection{Extensionality}
 
The  most essential axiom required to sustain an ontology of sets is the following: 


$$\textbf{Extensionality: }\;\;\;\;\; \boldsymbol{ \forall x \forall y ( x=y  \; \Leftrightarrow \;  \forall z(z\in x\Leftrightarrow z\in y))} .$$
 
\p Why do we choose  strong rather than weak bi-implications? For the first one the choice is clear: extensionality seeks to define the {\it meaning} of the expression ``$x=y$'' in terms of the elements of $x$ and $y$,  and this  needs to talk both about {\it truth} and {\it falsity}.  Notice that the use of this strong implications allows us to {\it interchange} the expression ``$x=y$'' with the expression on the right-hand side within any given formula.

 The second bi-implication is more interesting: recall that we want $x=y$ to express that both the !-extension and ?-extension of $x$ and $y$ coincide. The statement with a weak implication $  \forall z(z\in x\leftrightarrow z\in y))$ would only say that the !-extensions coincide.


\bigskip At this point it is instructive to define subsets:

$$x \subseteq y \abbr \forall z \;(z \in x \Rightarrow z \in y)$$

\p
Again, the use of the strong implication  makes sure that $x \subseteq y$ says that  both the !-extensions and ?-extensions of $x$ are included in that of $y$.  In particular, the Axiom of Extensionality can now be reformulated as follows: $\forall x \forall y ( x=y  \; \Leftrightarrow \;  x \subseteq y \land y \subseteq x)$. 

\subsection{Comprehension} The next most important axiom is Comprehension.\footnote{We  suppress mention of parameters  to simplify notation.}

$$\textbf{Comprehension:} \;\;\; \boldsymbol{ \forall u\exists x \forall y \;  (y\in x \Leftrightarrow y\in u \land\varphi(y))}$$

\p The use of the strong implication means that,  in accordance to Convention \ref{conny},  for any set $u$, the following is also a set:   $$x= \{y \in u : \varphi(y)\}.$$

\subsection{Classical Supersets}

The development of our theory is greatly simplified by considering an  axiom postulating  that every set is contained within a classical superset.  In Section \ref{intuition} we said that a set was {\it classical} if its $!$-extension and $?$-extension are the same. We do not know yet whether !-extensions and $?$-extensions are sets, but we can express that $C$ is a classical set with the sentence $\forall y \; {\circ} (y \in C)$, where $\circ$ is the classicality operator defined in Section \ref{defined}.

$$\textbf{Classical Superset:} \;\;\; \boldsymbol{ \forall x\: \exists  C \: (x \subseteq C \;\land \; \forall y \;{ \circ}(y\in C))}$$

This axiom is technically superfluous, since it can be proved to follow from the remaining axioms in a roundabout way.  However, adopting it at this stage allows us to frame the remaining axioms, and the theory in general, in a more intuitive way.  In particular, it allows us to confirm several properties of sets we postulated earlier.


\begin{Lem}[Cl. Superset + Comprehension] If $x$ is a set, then its !-extension and ?-extension are sets.  \label{extension} \end{Lem}

\begin{proof} Let $C$ be a classical set with $x \subseteq C$. Then

$$x^! \; = \; \{y \in C \; : \; !(y \in x)\}$$
$$x^? \; = \; \{y \in C \; : \; ?(y \in x)\}$$
 are both sets by the Comprehension axiom. \end{proof}

Since $x^!$ and $x^?$ are classical, {\it a posteriori} we see that a particularly convenient classical superset of $x$ is obtained by considering $x^! \cup x^?$. This is the  {\it smallest} classical set containing $x$ and we will   call this the  \emph{realm} of $x$:
$$\mathrm{rlm}(x) := x^! \cup x^?.$$





The next lemma elaborates on the meaning of equality and the subset relation between sets. 

\begin{Lem}[Cl. Superset + Extensionality + Comprehension]  \label{shwa}  $\;$

 \begin{enumerate}

\item $x \subseteq y \;\; \leqqq \;\; x^! \subseteq y^! \land x^? \subseteq y^? $

\item $x \not\subseteq y \;\; \leqqq \;\; \exists z (z \in x \land z\notin y) \;\; \leqqq \;\;  x^! \not\subseteq y^?$

\item $?(x \subseteq y) \;\; \leqqq \;\;   x^! \subseteq y^?$ 

\item $x = y \;\; \leqqq \;\;  x^! =  y^! \land x^? = y^?$

\item $x \neq y \;\; \leqqq \;\;   \exists z ((z \in x \land z\notin y) \lor (z \in y \land z\notin x)) \;\; \leqqq \;\; x^! \not\subseteq y^? \lor y^! \not\subseteq x^?$
 
\item $?(x = y) \;\; \leqqq \;\;  x^! \subseteq  y^? \land y^! \subseteq x^?$


\end{enumerate}

\end{Lem}
 
 \begin{proof} We will provide detailed proofs of 1 and 2 in order to illustrate how reasoning within $\BS$ works.  The remaining proofs are left to the reader.
 
 \begin{enumerate}

 \item We have the following sequence of $\BS$-provable bi-implications:

\noindent $x \subseteq y$     \hspace{0.2cm}$\Leftrightarrow$   \hspace{0.2cm} $ \forall z ( (z \in x \to z \in y)   \land (\no (z \in y) \to \no (z \in x))$     \hspace{0.2cm}$\stackrel{(*)}{\leqqq }$
 
 \medskip

\noindent  $\forall z ( (!(z \in x)  \to !(z \in y))  \land  (? (z \in x) \to ? (z \in y))$

 \medskip
 
\noindent  $\stackrel{(**)}{\Leftrightarrow}$      \hspace{0.2cm}$ \forall z (( z \in x^! \Rightarrow z \in y^!) \land (z \in x^? \Rightarrow z \in y^?)$    \hspace{0.2cm}   $ \Leftrightarrow$   \hspace{0.2cm} $x^! \subseteq y^! \land x^? \subseteq y^? $.
 
 \smallskip \p Here $(*)$ is due to the truth-functional definition of the implications  and the $!$ and $?$ operators, while $(**)$ is due to the definition of $x^!, y^!, x^?, y^?$ and the fact that these sets are classical (so $\to $ and $\Rightarrow$ are interchangeable). The first and last strong bi-implication is the definition of $\subseteq$.
 
 \item We have the following sequence of bi-implications:
 
\noindent $\no (x \subseteq y)$ 
    $\Leftrightarrow$      
 $ \no \forall z (z \in x \Rightarrow z \in y)   $  
    $\Leftrightarrow$    
 $   \exists z \no (z \in x \Rightarrow  z\in y)   $ 
    $\stackrel{(*)}{\leqqq }$    
   $   \exists z (z \in x \land z \notin y) $
 
 where $(*)$ is because we are only looking at the {\it truth} condition of $\Rightarrow$.  Further:
 
\noindent $   \dots $
  $\stackrel{(**)}{\leqqq }$
 $   \exists z (z \in x^! \land z \notin y^?) $
   $\Leftrightarrow$
  $  x^! \not\subseteq y^?$,
  where $(**)$ is again due to the fact that $!$ refers to truth and $?$ to falsity. \qedhere
 
 \end{enumerate}
 
 \end{proof}

\subsection{Replacement} The next axiom we consider is Replacement.\footnote{As before, we  suppress mention of parameters to simplify notation.} In $\ZFC$ the Replacement axiom tells us that the image of a set under a class function is itself a set.  It is not immediately clear how to generalize {\it class function}.  As a guiding principle we rely on the intuition that in practice,  mathematicians apply  Replacement  when there is a pre-determined  recipe by which each element of a given set is {\it replaced} by another element in a non-ambiguous way.  We will call such a recipe an {\it operation}:

\begin{Def} \label{operation} An {\it operation} is a formula $\varphi(x,y)$  such that \begin{enumerate}

\item $\varphi$ is classical, and
\item $\forall x \exists y(\varphi(x,y) \; \land \; \forall z \;(\varphi(x,z) \to !(y=z))$.

\end{enumerate}
 \end{Def}
The requirement on  $\varphi$   to be a classical formula reflects the notion that an operation describes a well-defined recipe for replacing  input $x$ with  output $y$. Likewise, the  ``$!$'' makes sure that the statement expressing that ``every input has at most one output'' is a classical sentence, since we would not know how to interpret a situation in which this statement   is both true and false.\footnote{One can check that this will occur, for example, with the formula $\varphi(x,y) \; \equiv \; !(y=a)$ where $a$ is any non-classical set.} 
$$\textbf{Replacement:} \;\;\; \boldsymbol{ ({\circ} \varphi  \land \forall x \exists y (\varphi(x,y) \land \forall z (\varphi(x,z) \to !(y=z)) )  \to }$$
 
\hspace{3.7cm} $\boldsymbol{\forall x \exists y \forall z (z \in y \Leftrightarrow \exists w (w \in x \land \varphi(w,z)))}$

\p After adopting this axiom, we can treat an operation $\varphi$ as a class function $F$, and use notation such as 
$$F[X] := \{y : \exists x  \;(x \in X  \land \varphi(x,y)\}$$ where $X$ is a set.  

\subsection{Pairing} The Pairing axiom is normally needed to get set-theory `started' and allow the definition of ordered pairs, relations, functions, and so on. Recall the discussion in Convention \ref{conny2}   that  we want  notation such as $\{u,v\}$ to stand for the {\it classical} set whose !-extension and ?-extension contain exactly the two objects $u$ and $v$.  This motivates the following

$$\textbf{Pairing:} \;\;\; \boldsymbol{ \forall u   \forall v  \exists x \forall y  ( y \in x \: \Leftrightarrow  \;   ( !(y=u) \lor !(y=v)))}$$

As discussed in Section \ref{intuition}, this falls in line with the intuition of an unordered pair $\{u,v\}$ as a classical set.  Referring to Lemma \ref{shwa}. (5) and (6), we now see that, if we had used the definition $\{x : x=u \lor x=v\}$ for the unordered pair,  then the ?-extension would have been the collection of all $x$ such that $x^! \subseteq u^?$ and $u^! \subseteq x^?$, or $x^! \subseteq v^?$ and $v^! \subseteq x^?$.  This is problematic since there is no upper bound on the size of $x^?$.  Indeed,  in Lemma \ref{properclass} we will prove that if there exists at least one incomplete set,  then $\{x : x=u \lor x=v\}$ is a proper class.

\subsection{Power Set} Suppose $u$ is any set and we look at  $\PP(u) := \{x : x \subseteq u\}$. Referring to Lemma \ref{shwa}, we again notice that $?(x\subseteq u) \;\; \leqqq \;\; x^! \subseteq u^?$, so the ?-extension of $\PP(u)$ consists of  sets $x$ such that $x^! \subseteq u^?$, but with no special requirement on $x^?$. Again, it should be intuitively clear that there is a proper class of possible $x$ satisfying this requirement.  As with Pairing, we instead opt for the following definition:

$$\PP^!(u) := \{x : \; !(x \subseteq u)\}$$

\p Now $\PP^!(u)$ is a classical set containing exactly those sets $x$ for which $x^! \subseteq u^!$ and $x^? \subseteq u^?$.  This  motivates the axiom:

$$\textbf{Power Set:} \;\;\; \boldsymbol{ \forall u  \exists v  \forall x   ( x \in v \: \Leftrightarrow  \; \;  !(x \subseteq u))    }$$


\subsection{The remaining axioms} We will now list the remaining four axioms since they are, mostly, non-problematic. 

$$\textbf{Union:} \;\;\; \boldsymbol{ \forall u \exists x \forall y (y \in x  \: \Leftrightarrow  \;  \exists z \: (y \in z \land z \in u)}$$
  After adopting this axiom  we can use the abbreviation $\bigcup u := \{x : \exists z \: (y \in z \land z\in u)\}$, and this coincides with Convention \ref{conny}.  

 $$\textbf{Infinity:} \;\;\; \boldsymbol{ \exists x \; ( \varnothing \in x \; \land \forall y (y \in x \to y \cup \{y\} \in x)}$$
$$\textbf{Foundation:} \;\;\; \boldsymbol{ \forall x  ( \forall y  ( y \in x^! \cup x^? \to \varphi(y) ) \to \varphi(x)) \;\; \to \;\; \forall u \varphi(u)}$$
 This axiom could more properly be called ``set induction schema''. It allows us to view the universe as being constructed by transfinite recursion, where each new level consists of those $x$ for which the {\it realm}  $x^! \cup x^?$ is a subset of the previous level. See \cite[Section 4.12]{HrafnThesis} for details.
 
\vbox{$$\textbf{Choice:} \;\;\; \boldsymbol{          \forall u ( \forall x (x \in u \to \exists y (y \in x)) \; \to }  \boldsymbol{ \; \exists f ( \dom(f)=u} $$ 
$$\hspace{2cm}\boldsymbol{\land \; \forall x  (x \in u \to f(x) \in x) ) )     }$$}

 This is a standard formulation of the axiom of choice, however, it requires the concept of a {\it function} which we have not properly defined yet.  This will be done in Section \ref{math}, however the current axiom is not required for that definition. 

\section{The anti-classicality axiom and $\BZFC$}  \label{bzfc}

None of the axioms of $\PZFC$ guarantee the existence of an inconsistent or incomplete set. In fact, the following should be clear:

\begin{Thm} $\PZFC \: + \: \forall x  (x^!=x^?)$ is equivalent to $\ZFC$. \end{Thm}

\begin{proof} If every set is classical, then there is no distinction between $\no$ and $\lnot$,  nor between $\to$ and $\Rightarrow$. Likewise, $!$ and $?$ can  be discarded.  The  Classical Superset Axiom is trivial.  So what remains of $\PZFC$ is precisely the collection of $\ZFC$ axioms. \end{proof}

Since we are interested in  theories with non-classical sets, we would like to adopt an axiom postulating their existence.  Now we seem to be faced with  a choice:  exactly {\it which} non-classical sets do we want to postulate the existence of? Following a  conservative approach,  we might want to require only that {\it at least one} inconsistent and  {\it at least one}  incomplete set exists: 

$$\exists x (x^! \not\subseteq x^?) \; \land \; \exists x(x^? \not\subseteq x^!)$$ 

On the other hand, a  maximality approach might lead us to postulate the existence of {\it as many non-classical sets} as possible, i.e.,  for {\it any} classical sets $u$, $v$, there exists a set $x$ whose !-extension is exactly $u$ and whose ?-extension is exactly $v$

$$\forall u \forall v (u \text{ and } v \text{ classical } \; \to \; \exists x \:(x^! = u \land x^?=v)).$$

An unexpected, yet surprisingly simple result now shows that in the presence of the other $\PZFC$-axioms, any choice we make is equivalent, since the weakest of them (the conservative one) already implies the strongest (the maximizing one).  Indeed, this can be viewed as a very central theorem on which the rest of our theory hinges in a crucial way.

\begin{Thm}[$\PZFC$] \label{aclaacla} Suppose there is an inconsistent and an incomplete set. Then for any classical sets $u,v$, there is $x$ such that $x^!=u$ and $x^?=v$. \end{Thm}

\begin{proof} From the assumption we have $a \in b \land a \notin b$ for some sets $a,b$ and also $\lnot (c \in d \lor c \notin d)$ for some sets $c,d$. Let us introduce the following abbreviation:
$$\phi_\bb \;\; \equiv \;\; a \in b$$
$$\phi_\nn \;\; \equiv \;\; c \in d$$
Note that $\phi_\bb$ and $\no \phi_\bb$ are both true, while $\phi_\nn$ and $\no \phi_\nn$ are both not true.

\p Let $u$ and $v$ be classical sets and define
$$x := \{z \in u \cup v\;  \;: \; \; z \in (u \cap v)  \; \lor \;(z \in (u\setminus v) \land \phi_\bb) \; \lor \; (z \in (v\setminus u) \land\phi_\nn)\}$$
Then we have: 
\begin{itemize} 
\item 
  $z\in x  $
 $\;\;\leftrightarrow\;\;$
  $z \in (u \cap v)  \; \lor \;(z \in (u\setminus v) \land\phi_\bb) \; \lor \; (z \in (v\setminus u) \land\phi_\nn)  $
 $\;\;\leftrightarrow\;\;$
 $z \in (u \cap v)  \; \lor \;z \in (u\setminus v)   $
 $\;\;\leftrightarrow\;\;$
  $z \in u$, and
  
  \item $z\notin x  $
 $\;\;\leftrightarrow\;\;$
  $z \notin (u \cap v)  \; \land \;(z \notin (u\setminus v) \lor \no \phi_\bb) \; \land \; (z \notin (v\setminus u) \lor \no \phi_\nn)  $
 $\;\;\leftrightarrow\;\;$
 $z \notin (u \cap v)    \; \land \; z \notin (v\setminus u)   $
 $\;\;\leftrightarrow\;\;$
  $z \notin v.  $ \end{itemize} 
   It follows that $z \in x^! \leqqq z \in u$ and $z \in x^? \leqqq z \in v$. This completes the proof. \end{proof}

\begin{Def} We let the {\it anti-classicality axiom} be the statement:
 $$\textbf{ACLA:} \;\;\; \boldsymbol{  \exists x (x^! \not\subseteq x^?) \; \land \; \exists x(x^? \not\subseteq x^!)}$$ 
  and consider the system 
 $$\BZFC \;\; \equiv \;\; \PZFC + \ACLA$$
 This is the main axiomatic system under consideration in this paper.
\end{Def}

\begin{Remark} \label{moregeneral} Some readers might be interested in a finer distinction and consider  a theory that has only incomplete but not inconsistent sets. The corresponding theory would then be    $\PZFC + \exists x (x^? \not\subseteq x^!) + \forall x (x^! \subseteq x^?)$.  A proof similar to the above would then show that this implies that for any classical sets $u,v$ with $u \subseteq v$, there is $x$ such that $x^! = u$ and $x^? = v$.  Likewise,  if one is interested in a theory   that has only inconsistent but not incomplete sets  one can look at $\PZFC + \exists x (x^! \not\subseteq x^?) + \forall x (x^? \subseteq x^!)$. This implies that for any classical sets $u,v$ with $v \subseteq u$, there is $x$ such that $x^! = u$ and $x^? = v$.  We will return to this finer distinction in Section \ref{modeltheory} in a specific context, but otherwise will not pursue it  in detail. 
\end{Remark}
 
 As a nice application of $\ACLA$,  we can show how to internally define truth values in a succinct way.  Let $\Omega := \PP^!(\{\varnothing\})$, i.e.,  the classical set containing sets $x$ such that $x^! \subseteq \{\varnothing\}$ and $x^? \subseteq \{\varnothing\}$.  There are four possible combinations for such $x$, and $\ACLA$ guarantees us that all four of them exist, and are members of $\Omega$.  We can give them names as follows:

\begin{itemize}

\item $x := 1\;\;\;\;$ if $\;\;\;\;x^! = x^? = \{\varnothing\}$
\item $x:=0\;\;\;\;$ if  $\;\;\;\;x^! = x^? = \varnothing$
\item $x:=\nn\;\;\;\;$ if $\;\;\;\;x^! = \varnothing$ and $x^? = \{\varnothing\}$
\item $x:=\bb\;\;\;\;$ if $\;\;\;\;x^! = \{\varnothing\}$ and $x^? = \varnothing$
\end{itemize}
Then $\Omega = \{1, \bb,\nn,0\}$ is called the {\it set of truth values}, and for any formula $\varphi$, we can define the {\it truth value} of $\varphi$ by:

$$\llbracket \varphi \rrbracket := \{\varnothing : \varphi\} $$We now have that $\varnothing \in \llbracket \varphi \rrbracket$ is true precisely if $\varphi$ is true, and false precisely if $\varphi$ is false. In other words, for any formula $\varphi$, from the point of view of the meta-theory, i.e.,  $\llbracket \;  \varnothing \in \llbracket \varphi \rrbracket \;  \rrbracket \; = \; \llbracket \varphi \rrbracket.$  

\section{Mathematics in $\BZFC$} \label{math} 

In a typical set theory textbook, the introduction of the $\ZFC$-axioms is usually followed up by developing the tools needed to sustain modern mathematics, e.g., ordered pairs, Cartesian products, relations, functions, induction and recursion,  ordinals, cardinals and their arithmetic.

In general, many non-trivial questions arise in the context of $\BZFC$, such as the exact nature of functions,  images and pre-images,  cardinality,  cardinal arithmetic and much more.  Many of these questions deserve separate investigation. In this paper we will present only  as much mathematical formalism as is necessary for the subsequent sections.   Readers  more interested in the foundational results can safely skip   to the next Section \ref{model}.

\bigskip  First, we would like to define an {\it ordered pair}  $(a,b)$   in such a way that $$(a,b) = (a',b') \; \Leftrightarrow \;(a=a\ \land b=b'). \hspace{2cm} (*)$$ The task is less trivial than it appears, since the falsity condition $(a,b) \neq (a',b') \; \leftrightarrow \;  (a \neq a' \lor b \neq b')$ refers to the native $\neq$ relation which is governed by the extensionality axiom. The  Kuratowski pair will not suffice for this purpose since this  defines a classical set.  But we can define ordered pairs in a round-about way as follows:

\begin{itemize}
\item First, say that the {\it classical} ordered pair $\left<u,v\right>$ is $\{\{u\}, \{u,v\}\}$.
\item Then, for sets $a$ and $b$ define $(a,b) \; := \; \{ \left<x,0\right> : x \in a\} \cup \{ \left<y,1\right> : y \in b\}$, where $0$ is $\varnothing$ and $1$ is $\{\varnothing\}$.
\end{itemize}
Here, $(a,b)$ is essentially the disjoint union of copies of $a$ and $b$, so the non-classical structure of the sets is unaffected.  Showing that this definition indeed satisfies $(*)$ is somewhat lengthy, so we leave the details to the reader. 
 We should note that there are  other ways to encode  ordered pairs,  for example see   \cite[Appendix A]{HrafnThesis}.  The exact method is inconsequential as long as condition $(*)$ is fulfilled. 

 The {\it Cartesian product} $A \times B := \{(a,b) : a \in A \land b \in B\}$  is defined as usual.  It is easy to see that $(A \times B)^! = A^! \times B^!$ and $(A\times B)^? = A^? \times B^?$, so if $A$ and $B$ are classical,  $A \times B$ is as well.
%
A binary {\it relation} between $A$ and $B$ is any  $R \subseteq  A \times B$.  We note that $R$ may fail to be classical  even if $A$ and $B$ are.  We can define the domain and range of a relation by stipulating that $\dom(R) = \{x : \exists y ((x,y) \in R)\}$ and $\ran(R) = \{y : \exists x ((x,y) \in R)\}$.  It then follows that $\dom(R) \subseteq A$ and $\ran(R) \subseteq B$. The domain and range will generally be non-classical if $R$ is non-classical.  The following definition deserves special attention:

\begin{Def}  \label{eqrel} $E \subseteq X \times X$ is called an {\it equivalence relation} if for all $x,y,z$ in  rlm$(X)$ (i.e.,  in $X^! \cup X^?$) we have:
\begin{enumerate}

\item $(x,x) \in E$.
\item $(x,y) \in E \Leftrightarrow (y,x) \in E$.
\item $(y,z) \in E \; \to \; ((x,y) \in E \Leftrightarrow (x,z) \in E)$.
\end{enumerate} \end{Def}
Notice that this tells us more than $E$ being reflexive, symmetric and transitive; in addition,  3 implies that if $y$ and $z$ are $E$-related, then they are  indistinguishable with respect to being $E$-related to another element $x$.  If we now define the {\it $E$-equivalence class} of an element $x \in X$ by
$$[x]_E := \{y : (x,y) \in E\}$$ we obtain the following:

\begin{Lem} \label{eqrel2} For all $x,y \in {\rm rlm}(X)$ we have  $(x,y) \in E \; \Leftrightarrow \; [x]_E = [y]_E$.
\end{Lem}

\begin{proof} The positive  equivalence  $(x,y) \in E \; \leftrightarrow \; [x]_E = [y]_E$ is obvious.  Also, it is clear that if $(x,y) \notin E$, then by definition $x \notin [y]_E$, while  $(x,x) \in E$ implies that $x \in [x]_E$.  Therefore there is a witness $x$ which is in  $[x]_E$ and not in $[y]_E$,  so $[x]_E \neq [y]_E$ by extensionality.

\p Now, suppose $[x]_E \neq [y]_E$.  Then there exists $z$ such that $z \in [x]_E$ and $z \notin [y]_E$, or vice versa. Without loss assume the former, so by definition we have $(z,x) \in E$ and $(z,y) \notin E$.  But now the strong bi-implication in condition 3 gives us $(x,y) \notin E$, as required. \end{proof}

Next we look at the notion of a {\it function}.  It is not entirely straightforward how the concept of a {\it non-classical function} should be understood,  nor how the related concepts of {\it image}, {\it pre-image} etc.  of such a function mean.  However,  for the purposes of this paper,  it will suffice to look only at {\it classical} functions:

\begin{Def} If $A$ is a classical set,  a {\it classical function} from $A$ to $B$ is a relation $R \subseteq A \times B$ which is classical,  satisfies $\dom(R) = A$, and 

$$\forall x \forall y \forall  z ((x,y) \in f \land (x,z) \in f) \; \to !(y=z))$$ \end{Def}


Classical functions take sets as inputs and generate other sets as outputs in a unique way, just as ordinary functions do.  Note, however,  that even if $f$ is classical, the inputs and outputs need not be classical, so  it is dangerous and potentially misleading to use notation like ``$f(x) = y$''.  For example,  suppose $x$ and $y$ are  such that $(x,y) \in f$ and $y\neq y$.  If we write  ``$f(x) = y$'' we should also write ``$f(x) \neq y$''.  But this should not be confused with ``$(x,y) \notin f$''.

Recall that we already  introduced the  concept of an {\it operation} (Definition \ref{operation}).  If $\varphi(x,y)$ is an operation and $A$ a classical set, then $\{(x,y) : x \in A \land \varphi(x,y)\}$ is a classical function.\footnote{Here we use Replacement followed by Comprehension to ensure that this object is a set and not just a proper class.} Conversely, if  $f$ is a classical function with (classical) domain $A$, and we fix an arbitrary  $y_0$,  then the following formula is an operation which coincides with $f$ on  $A = \dom(f)$:
$$\varphi(x,y) \;   \equiv  \;  ((x \in A \to (x,y) \in f)  \land (( x \notin A \to !(y=y_0))$$
When we talk about {\it isomorphisms} in the next sections,  we will be referring to classical functions when the structures are sets, or operations when they are proper classes.

\bigskip We now turn our attention to the notions of  {\it ordinal}, {\it induction} on ordinals, and {\it recursive definitions}.  
Informally,  ordinals will be  {\it classical} {\it transitive} sets,   totally ordered by $\in$,  and containing only  {\it classical} sets as members:

\vbox{\begin{Def} $\alpha$ is an {\it ordinal} iff
\begin{itemize}
\item $\alpha$ is classical.
\item $\forall \beta \: ( \beta \in \alpha \; \to \; \beta$ is classical$)$.
\item $\forall \beta \: (\beta \in \alpha \; \to \; \beta \subseteq \alpha)$.
\item $\forall \beta  \:\forall \gamma \: ((\beta \in \alpha \land \gamma \in \alpha) \; \to \; (\beta \in \gamma \lor \gamma \in \beta \lor \beta = \gamma))$.
\end{itemize}
\end{Def} }
All  facts about ordinals known from $\ZFC$ are also true for this definition, because any formula $\varphi$ that refers only to ordinals and their members, will be a classical formula. For example, one can show that elements of ordinals are themselves ordinals,  that ordinals are unique up to isomorphism,  that any well-founded structure is isomorphic to an ordinal,  and so on. The same holds for the transfinite induction principle on ordinals, stating that for any formula $\varphi$:

 \vbox{
$$\left( \forall \alpha (\alpha \text{ is an ordinal } \to\;  (\forall \beta (\beta \in \alpha \land \varphi(\beta) )\to \varphi(\alpha) \right)$$
$$\to \;  \forall \alpha (\alpha \text{ is an ordinal } \to \varphi(\alpha)) $$}

From this, the recursion principle follows by usual methods again: 

\begin{Lem}[Recursion Principle] Let $G$ be any {\it operation}. Then there exists a unique operation $F$, such that $F(x,y)$ implies that $x$ is an ordinal,  and such that for every $\alpha$ we have $$F(\alpha) = G(F \till \alpha).\footnote{Here we have abused notation somewhat, but recall that $F$ and $G$ are, by definition,  given by classical formulas, so this abuse of notation is consistent with its usage in classical $\ZFC$.}$$ \end{Lem}

\bigskip
We end this section by providing the promised proof that a  careless translation of singletons, pairs and power sets would result in proper classes. 

\begin{Lem}[$\BZFC$] \label{properclass} For sets $u, u_1, \dots u_n$, the collections $\{y : y \subseteq u\}$, $\{y : y = u\}$ and $\{y : y=u_1 \lor \dots \lor y=u_n\}$ are proper classes. \end{Lem}

\begin{proof} We will only prove the first case since the others are similar.  We know  that $\{x : \lnot (x \in x)\}$ is a proper class, and from this it easily follows that the collection $\Clas$ of all classical sets forms a proper class. 

\p Suppose there is a set $X = \{y : y \subseteq u\}$.  By Lemma \ref{extension} we know that $X^?$ is also a set.  By Lemma \ref{shwa} (3), we know that $y \in X^? \Leftrightarrow y^! \subseteq u^?$.  Consider the operation $$F: \begin{array}{rll} X^? & \to & \Clas\\ y & \mapsto & y^?\end{array}$$ Formally, this operation is given by $\varphi(y,w) \; \equiv \; !(w=y^?)$ which is  a classical formula and satisfies the conditions for the Replacement axiom.  Moreover, by $\ACLA$ the operation is surjective,  i.e., for every classical $w$ there is $y \in X^?$ such that $\varphi(y,w)$ holds (take $y$ with $y^! = u^?$ and $y^? = w$). But then $F[X^?] = \Clas$,  and since $X^?$ is a set,  $\Clas$ should be a set, which is a contradiction.\end{proof}

The above argument  is somewhat informal,  so readers may wonder what it means to says that something is a proper class, or what  a proof by contradiction means in  the paraconsistent setting. Formally, what we have proven is that if a set $X$ as   above exists, then $\bot$.

\section{A model of $\BZFC$} \label{model}

In this section we construct a   T/F-model for $\BZFC$ starting from  $\ZFC$.  Usually this would yield a {\it relative consistency} proof, i.e., a proof that if $\ZFC$ is consistent then $\BZFC$ is consistent.  Of course $\BZFC$ is, by design,  {\it inconsistent}, so instead we will talk of {\it non-triviality:}

\begin{Def} A $\BS$-theory $\Gamma$ is called {\it non-trivial} if $\Gamma \not\vdash_\BS \bot$.  
\end{Def}



\begin{Def}[$\ZFC$] By induction on ordinals we define: \label{defW}

\begin{itemize}

\item $W_0 = \varnothing$
\item $W_{\alpha+1} := \PP(W_\alpha) \times \PP(W_\alpha)$
 \item $W_{\lambda} = \bigcup_{\alpha<\lambda}W_\alpha$ for limit $\lambda$
\item $\displaystyle \IW := \bigcup_{\alpha \in \textup{Ord}}W_\alpha. $

\end{itemize}
 The positive and negative interpretations of $\in$ and $=$ are given by:

\begin{itemize}

\item $(a,b) \in^+ (c,d)$ iff $(a,b) \in c$
\item $(a,b) \in^- (c,d)$ iff $(a,b) \notin d$
\item $(a,b) =^+ (c,d)$ iff $(a,b) = (c,d)$ 
\item $(a,b) =^- (c,d)$ iff $\exists z \in a \setminus d$ or $\exists z \in c \setminus b$.

\end{itemize}
\end{Def}
 The following  properties of $\IW$ are easily verifiable. 

\begin{Lem}[$\ZFC$] $\;$ \begin{enumerate} \label{helpful}
\item If $\alpha<\beta$ then $W_\alpha \subseteq W_\beta$.
\item If $(a,b) \in W_\alpha$, $a' \subseteq a,$ and $b' \subseteq b$ then $(a',b') \in W_\alpha$.
\item $x \in \IW$  iff $x = (a,b)$ for some $a,b \subseteq \IW$. 
\end{enumerate}
\end{Lem}

\begin{proof} Induction on the definition.  \end{proof}

Now $(\IW,\in^+, \in^-, =^+,=^-)$ may be considered  a T/F-model in the language of set theory, except that   $\IW$ is a proper class.  Therefore,   ``$\IW \models  \varphi$'' must be understood via  {\it relativization}:  for every $\varphi$ we define   $\varphi^{\IW, T} $ and $\varphi^{\IW,F}$ by syntactic induction,  generalizing from Definition \ref{mainsemantics} in the obvious way (we leave the details to the reader).  For a theory $\Gamma$,  $\IW \models \Gamma$ means that for every $\varphi$ in $\Gamma$, there is a proof of $\varphi^{\IW,T}$.

\begin{Thm}[$\ZFC$] $\IW \models \BZFC$.   \label{W} 
\end{Thm}

\begin{proof} We will  prove Extensionality and Comprehension in some detail  and leave the rest to the reader. 

\p Written out in full, the relativization (Extensionality)$^{\IW,T}$ reads as follows

$$\forall (a,b) \in \IW \; \forall (c,d) \in \IW:$$
$$(a,b) =^+ (c,d) \;\; \leqqq\;\; \forall z\in \IW  ((z \in^+ (a,b) \leqqq z \in^+ (c,d)) \land (z \in^- (a,b) \leqqq z \in^- (c,d)) $$
$$ \land \; (a,b) =^- (c,d)  \;\; \leqqq\;\;\exists z \in \IW ((z \in^+ (a,b) \land z \in^- (c,d)) \lor (z \in^+ (c,d) \land z \in^- (a,b))  $$

\p Assume $(a,b)$ and $(c,d)$ are arbitrary, and we show that the two equivalences hold.  For the first we have \medskip 

 \medskip $ \forall z\in \IW  ((z \in^+ (a,b) \leqqq z \in^+ (c,d)) \land (z \in^- (a,b) \leqqq z \in^- (c,d))$
 
 \medskip  $\stackrel{(*)}{\leqqq }\;\;$
  $\forall z\in \IW  ((z \in a \leqqq z \in c) \land (z \notin b \leqqq z \notin d)).$

 \medskip  $\stackrel{(**)}{\leqqq }\;\;$
 $\forall z  ((z \in a \leqqq z \in c) \land (z \notin b \leqqq z \notin d))$

 \medskip $\leqqq\;\; $  $a=c$ and $b=d$

 \medskip  $\leqqq\;\;$   $(a,b) =^+ (c,d)$

 \noindent where $(*)$  refers to the definition of $\in^+$ and $\in^-$, and $(**)$ is because  $\IW$ is ``transitive'' in the sense that  $a,b,c,d \in \IW$ only contain sets which are also in $\IW$.

  \p For the second  equivalence we have
  
   \medskip $\exists z \in \IW((z \in^+ (a,b) \land z \in^- (c,d)) \lor (z \in^+ (c,d) \land z \in^- (a,b))$

  \medskip$\leqqq\;\;$  $\exists z(( z \in a \land z \notin d) \lor (z \in c \land z \notin b))$

  \medskip$\leqqq\;\;$  $(a,b) =^- (c,d)$

  \p Next we look at (Comprehension)$^{\IW,T}$ which is the following statement:\footnote{Again we suppress parameters, noting that if there are parameters these are understood to be in $\IW$.}

$$\forall (a,b) \in \IW \; \exists (c,d)  \in \IW\;  \forall z   \in \IW\:$$ $$ ( (z \in^+ (c,d) \leqqq (z \in^+ (a,b) \land \varphi(z)^{\IW,T}) \land (z \in^- (c,d) \leqqq (z \in^- (a,b) \lor \varphi(z)^{\IW,F}))$$

\p Suppose $(a,b)\in \IW$ and $\varphi$ is given.  Define $$c := \{z \in a : \varphi(z)^{\IW,T}\}$$ $$d := \{z \in b : \lnot \varphi(z)^{\IW,F}\}\footnote{An equivalent definition is $c = \{z \in a : \IW \models \: !\varphi(z)\}$ and $d = \{z \in b : \IW \models \: ? \varphi(z)\}$.}$$
Then $(c,d) \in \IW$ by construction and for all $z$ we have $z \in c \leqqq (z \in a \land \varphi(z)^{\IW,T})$ and $z \notin d \leqqq (z \notin b \lor  \varphi(z)^{\IW,F})$. This is exactly the statement above as needs to be proved. \end{proof}

\begin{Cor} If $\ZFC$ is consistent then $\BZFC$ is non-trivial. \end{Cor}

We have thus shown that $\BZFC$ is essentially not a more problematic theory than $\ZFC$.  In addition,  the canonical model $\IW$ contains a natural copy of the classical universe $V$:

\begin{Def}[$\ZFC$] \label{vcheck} For every $x \in V$,  inductively define $$\check{x}  :=  ( \{ \check{y} : y \in x\}, \;  \{ \check{y} : y \in x\}).$$  Also let $\check{V} := \{\check{x} : x \in V\}$. 
\end{Def}

\begin{Lem}[$\ZFC$]  \label{vcheckcheck} The mapping

$$i: \begin{array}{ccl}V & \rightarrow & \check{V} \subseteq \IW \\ x & \mapsto &\check{x} \end{array}$$
 is an isomorphism between $(V,\in, \notin,=,\neq)$ and $(\check{V},\in^+,\in^-,=^+,=^-)$.

 \end{Lem}

\begin{proof} Easy consequence of the definitions.
\end{proof}
 

\section{Hereditarily classical sets.}  \label{translation}

Starting in $\BZFC$, we can also construct a natural model of $\ZFC$: this is the class of ``hereditarily classical'' sets.

\begin{Def}[$\BZFC$] By induction on ordinals\footnote{Recall from Section \ref{math} that the recursion principle is valid in $\BZFC$.} define:

\begin{itemize}

\item $\HCL_0 = \varnothing$
\item $\HCL_{\alpha+1} := \{X \subseteq \HCL_\alpha : X$ is classical$\}$\footnote{Here we technically use  Power Set to first construct $\PP^!(\HCL_\alpha)$ and then Comprehension to select the classical sets.} 
 \item $\HCL_{\lambda} = \bigcup_{\alpha<\lambda}\HCL_\alpha$ for limit $\lambda$
\item $\displaystyle \IHCL := \bigcup_{\alpha \in \textup{Ord}}\HCL_\alpha. $

\end{itemize}
\end{Def}

$\IHCL$ is a transitive proper class,  and for all $x$,  we have that $x \in \IHCL$ if and only if $x$ is classical and $x \subseteq \IHCL$. 
Again, the notation $\IHCL \models \varphi$ and $\IHCL \models \Gamma$ refers to {\it relativization} where  quantification is restricted to range over $\IHCL$. 

\begin{Thm}[$\BZFC$] $\IHCL \models \ZFC$.  \label{HCL}
\end{Thm}

\begin{proof} It is clear that $\IHCL \models $ (every set is classical),  so by the discussion in Section \ref{bzfc} it suffices to show that $\IHCL \models \PZFC$.  Most of the axioms are straightforward since none of them postulate the existence of non-classical sets without assuming the existence of non-classical sets.

\p We will only show (Comprehension)$^\IHCL$. Suppose $u$ is a herditarily classical set and $\varphi$ any formula.   It is enough to show that $x := \{y \in u : \varphi^{\IHCL}(y,a_1, \dots, a_n)\}$ is a hereditarily classical set if the parameters $a_1,\dots, a_n$ are hereditarily classical.\footnote{Here we are explicit about parameters $a_1, \dots, a_n$ since otherwise $\varphi$ might fail to be a classical formula.}  Since $x \subseteq u$ and $u \subseteq \IHCL$ we know that $x \subseteq \IHCL$, so it remains to show that $x$ itself is classical.  But from the fact that $y, a_1, \dots a_n$ are classical sets, it follows by an easy induction that $\varphi$ is a classical formula, i.e., $!\varphi \leqqq ? \varphi$. It follows that $x^!=x^?$, so $x$ is classical. \end{proof}

\begin{Cor} If $\BZFC$ is non-trivial, then $\ZFC$ is consistent. \end{Cor}
In analogy to Definition \ref{vcheck}, we can now consider an embedding from the universe of $\BZFC$ (which we also refer to as $V$, hopefully not leading to confusion)  to a natural copy of it within $\IHCL$.

\begin{Def}[$\BZFC$] \label{DhatW}  For every set $x$, inductively define: 

$$\hat{x} := ( \{\hat{y} : y \in x^!\}, \{\hat{y} : y \in x^?\}).$$ 

\p  Also let $\hat{V} := \{\hat{x} : x \in V\}$.  For $\hat{x},\hat{y} \in \hat{V}$ define 

\begin{itemize}
 \item $ \hat{x} E^+ \hat{y}$ $\leqqq$ $x\in y$ 
 \item $ \hat{x} E^- \hat{y}$ $\leqqq$ $x\notin y$
 \item $ \hat{x} \approx^+ \hat{y}$ $\leqqq$ $x= y$
  \item $ \hat{x} \approx^- \hat{y}$ $\leqqq$ $x\neq y$
 \end{itemize}

\end{Def}
Now $\hat{x}$ are hereditarily classical sets, $\hat{V}$ is a proper class of hereditarily classical sets, and $E^+, E^-,\approx^+, \approx^-$ are hereditarily classical, class-sized binary relations on $\hat{V}$.

\begin{Lem} [$\BZFC$] The mapping \label{hatW}

$$j: \begin{array}{ccl}V & \rightarrow & \hat{V} \subseteq \IHCL \\ x & \mapsto &\hat{x} \end{array}$$
 is an isomorphism between $(V,\in, \notin, =,\neq)$ and $(\hat{V},E^+,E^-,\approx^+,\approx^-)$. 
 
 \end{Lem}
 
 \begin{proof} Follows immediately from the definitions. \end{proof}
 
 \section{Bi-interpretability} \label{biinterpretability}
 
We already know that $\ZFC$ in $\BZFC$ are {\it mutually interpretable}, i.e., each theory can construct a natural model for the other.  In fact we prove more than that:
 
 \begin{Thm} \label{bi-int} $\BZFC$ and $\ZFC$ are bi-interpretable. \end{Thm}
  
Here the term {\it bi-interpretability} may either be understood semantically,  saying that the  model constructions interpreting the other theory are reversible up to isomorphism, or syntactically, saying that a sentence 
is a consequence of one theory if and only if its translation (relativisation) is a consequence of the other,  and vice versa, in the corresponding logics.  Theorems \ref{ee} and \ref{eee} together with results from the previous sections concern the semantic understanding while Corollaries \ref{ff} and \ref{fff} refer to the syntactic one.

  \bigskip 
  In Lemma \ref{vcheckcheck} we defined an isomorphism  $i: x \mapsto \check{x}$ from $V$ and $\check{V}$, so it remains only to show that  $\check{V}$ is the same thing as what $\IW$ believes $\IHCL$ to be (provably in $\ZFC$). 
  
  \begin{Lem}[$\ZFC$] \label{ee} $\IHCL^\IW = \check{V}$. \end{Lem}
  
  \begin{proof}  To show $\supseteq$  take $\check{x} \in \check{V}$. By construction $\check{x} = ( \{\check{y} : y \in x\}, \{\check{y} : y \in x\})$.  Inductively, we may assume that each $\check{y}$ appearing above is in $\IHCL^\IW$, therefore $\IW \models \check{x} \subseteq \IHCL$. Moreover, $\check{x}$ has the same !-extension as ?-extension, therefore $\IW \models (\check{x}$ is classical). Together, this implies $\IW \models \check{x} \in \IHCL$.
  
 \p To show $\subseteq$, suppose $x \in \IHCL^\IW$, i.e., $x \in \IW$ and $\IW \models (x$ is hereditarily classical). Let $a,b \in \IW$ be such that $x=(a,b)$. Since $\IW \models x$ is classical, in particular $\IW \models x^! = x^?$,  therefore $a=b$.  Moreover, $\IW \models x \subseteq \IHCL$ which implies that $a \subseteq \IHCL^\IW$.  Inductively, we may assume that every element in $a$ is of the form $\check{z}$ for some $z \in V$. Let $\tilde{a} := \{z : \check{z} \in a\}$. Then $x= ( \{\check{z} : z \in \tilde{a}\}, \{\check{z} : z \in \tilde{a}\}) = (a,a)$ by construction.  But then, by definition, 
 $x = ({\tilde{a}}\check{)} $. 
   \end{proof}
   
  \begin{Cor} \label{ff} $\ZFC \vdash \varphi\;\;$ iff $\;\; \BZFC \vdash_\BS \left( \IHCL \models \varphi \right)$. \end{Cor}
  
 \begin{proof} If $\ZFC \vdash \varphi$ then clearly  $\BZFC \vdash_\BS \left( \IHCL \models \varphi \right)$.  Suppose  $\BZFC_\BS \vdash \left( \IHCL \models \varphi \right)$.  Then  $\ZFC \vdash \left( \IW \models \left( \IHCL \models \varphi \right) \right)$,  i.e., $\ZFC \vdash (\IHCL^\IW  \models \varphi )$. But by the above we have $\IHCL^\IW = \check{V} \cong V$, so $\ZFC \vdash \varphi$. \end{proof}
  
For the reverse direction, if we  start in $\BZFC$ we already have the $j: x \to \hat{x}$ from Definition \ref{DhatW}, so it remains to show that $\hat{V}$ is the same as what   $\IHCL$ believes $\IW$ to be, provably in $\BZFC$.
 
 
   \begin{Lem}[$\BZFC$] \label{eee} $\IW^\IHCL = \hat{V}$. \end{Lem}
  
  \begin{proof} Suppose $\hat{x} \in \hat{V}$.  Then $\hat{x} = ( \{\hat{y} : y \in x^!\}, \{\hat{y} : y \in x^?\})$,  and since every $\hat{y}$ appearing here is in $ \hat{V}$, inductively we can assume that it is also in $\IW^\IHCL$.  By definition $\hat{x}$ is hereditarily classical and we have $\IHCL \models (\hat{x} = (a,b)$ and $a,b \subseteq \IW)$.  By Lemma \ref{helpful} (3) applied in $\IHCL$, this means $\IHCL \models \hat{x} \in \IW$.
  
 \p Conversely,  suppose $w \in \IW^\IHCL$.  Again by Lemma \ref{helpful} (3) we know that  $\IHCL \models (w = (u,v)$ and $u,v \subseteq \IW)$. Then $u,v \subseteq \IW^\IHCL$ so, inductively, we have $u,v \subseteq \hat{V}$.  Therefore, let $\tilde{u} = \{y : \hat{y} \in u\}$ and $\tilde{v} = \{y : \hat{y} \in v\}$. Then $\tilde{u}$ and $\tilde{v}$ are classical sets (because $u,v$ are), so by  $\ACLA$ there exists a set $x$ such that $x^!  = \tilde{u}$ and $x^? = \tilde{v}$.  Then we have 
 $$\hat{x}  \;= \;  (\{\hat{y} : y \in x^!\} , \{\hat{y} : y \in x^?\}) \;= \;  (\{\hat{y} : y \in \tilde{u}\} , \{\hat{y} : y \in \tilde{v}\}) \;= \; (u,v) \; = \; w$$
 This shows that $w \in \hat{V}$.   \end{proof}

  \begin{Cor} \label{fff} $\BZFC \vdash_\BS  \varphi\;\;$ iff $\;\; \ZFC \vdash \left( \IW \models \varphi \right)$. \end{Cor}

 \begin{proof} If $\BZFC \vdash_\BS \varphi$ then clearly   $\ZFC \vdash \left( \IW \models \varphi \right)$.  Suppose  $\ZFC \vdash  \left( \IW \models \varphi\right)$.  Then $\BZFC \vdash_\BS \left( \IHCL \models \left( \IW \models \varphi\right)\right)$, so $\BZFC \vdash_\BS \left( \IW^\IHCL \models \varphi \right)$. But  in $\BZFC$ we have that $\IW^\IHCL = \hat{V} \cong V$, giving us $\BZFC \vdash_\BS \varphi$. \end{proof}

\section{Tarski semantics in $\PZFC$} \label{modeltheory}

A discussion that sometimes arises in the philosophy of logic is the extent to which the meta-theory affects the formal theory under consideration. For example,  the recursive definition of the negation and disjunction in Tarski semantics means that
$$\M \models \varphi \lor \lnot \varphi \;\;\;\; \text{iff} \;\;\;\; \M \models \varphi \text{ or } \M \not\models \varphi.$$
Thus  {\it excluded middle} is a tautology in classical logic only because {\it excluded middle} is covertly assumed to hold  in the meta-theory.  Likewise,  since we have
$$\M \models \varphi \land \lnot \varphi \;\;\;\; \text{iff}  \;\;\;\; \M \models \varphi \text{ and } \M \not\models \varphi,$$
a contradiction $\varphi \land \lnot \varphi$ is classically unsatisfiable only because  $\M \models \varphi$ and $\M \not\models \varphi$ cannot both be true in the meta-theory.
 
 Of course,  non-classical logics may be set up by considering alternative (non-Tarski) semantics, so that the formal system in question  exhibit  non-classicality while the meta-theory remains classical. Such is the case, for example,  with Kripke semantics for intuitionistic logic, and  the T/F-models for $\BS$ from Section \ref{logic}.  But this does not shed any light on the question how the logic generated by {\it Tarski} semantics (in a way,  the most fundamental  and semantics) is affected by the logic that holds in the meta-theory.
 
  Questions of this sort have been considered, for example, by  Shapiro in \cite[Chapters 6 and 7]{Shapiro-Varieties} or by Bacon in \cite{Andrew-Bacon},  and some related concerns about logical pluralism are voiced in \cite{waterplants, RobertP}.  No doubt, early logicians like G\"odel and Tarski  would have been aware of this issue as well.  Nevertheless, although the phenomenon of meta-logic affecting formal logic seems to be a well-known concern in philosophy of language and theories of truth,  we do not know of a rigorous {\it mathematical} analysis of this phenomenon,  certainly not in the context of paraconsistency.  Indeed, such an analysis would require a sufficiently developed theory of paraconsistent 
 mathematics to begin with, in which the relevant model-theoretic definitions and proofs might be formalised and carried out. Now that we have $\BZFC$ at our disposal, we are in  a position to do exactly that.
 
 \bigskip
 The setup is as follows: in  $\PZFC$,  we  can define  Tarski semantics by employing the usual induction definition (although predicates are interpretations  non-classically). This gives rise to a semantic  consequence relation, which we shall denote by $\models$,  and  one can study the {\it logic} which is sound and complete with respect to this relation. It turns out that in $\BZFC$ (i.e., in the presence of the  Anti-Classicality Axiom), this logic is precisely $\BS$ (Theorem \ref{tarski}).  On the other hand,  in $\PZFC \; + $ ``all sets are classical'', it is simply classical logic. 
 
  In fact, we can be more specific by looking at only the paraconsistent, or only the paracomplete, fragment of the Anti-Classicality Axiom.  In $\PZFC$, one can then show that the logic generated by Tarski semantics satisfies the exact same level of non-classicality as we assume to hold in the meta-theory (Theorems \ref{tt1} and \ref{tt2}).

  \begin{Def}[$\PZFC$]  \label{weirdsemantics} The syntax of first order logic is assumed to be coded by  {\it classical} sets.  A Tarski model $\M$ consists of a {\it classical set} $M$ as a domain, and (not necessarily classical) interpretations for all constant and relation symbols.  For simplicity we will leave out  function symbols.  Inductively we define: 

\begin{enumerate}

\item $\M \models (t=s)[a,b] \iff a=b$.

\item $\M \models R(t_1, \dots, t_n)[a_1, \dots, a_n] \iff R^\M(a_1, \dots a_n)$.

 \item $\M \models  \no \varphi \iff   \M \not\models \varphi$. \footnote{This is an abbreviation of $\no ( \M \models \varphi)$.}
 
 \item $\M \models  \varphi \land \psi   \iff \M \models \varphi$ and $\M \models \psi$.

 \item $\M \models  \varphi \lor \psi   \iff \M \models \varphi$ or $\M \models \psi$.

 \item $\M \models  \varphi \to \psi   \iff$ $( \M \models \varphi  \to   \M \models \psi)$.

 \item $\M \models  \varphi \leqqq \psi   \iff$ ($\M \models \varphi   \leqqq \M \models \psi)$.

  \item $\M \models  \exists x \varphi(x)    \iff \M \models  \varphi[a]$ for some $a \in M$.

  \item  $\M \models  \forall x \varphi(x)    \iff \M \models  \varphi[a]$ for all $a \in M$.
 
 \item  $\M \models  \bot \iff \bot$.

\end{enumerate}


\p For a set of formulas $\Sigma$ and a formula $\varphi$,  the {\it Tarski semantic consequence} relation is defined by $$\Sigma \models \varphi$$ if for every Tarski model $\M$, we have $(\M \models \Sigma \; \to \; \M \models \varphi).$
\end{Def}


For clarity, we will always refer to the above  as {\it Tarski models} and   {\it Tarski semantics}, to keep them apart from T/F-semantics from Definition \ref{mainsemantics}. We will also reserve the notation $\models$ for the Tarski satisfaction and consequence relation and use $\models^T$ and $\models^F$ when refering to T/F-semantics.

\begin{Remark} $\;$ 
\begin{enumerate}

\item
Definition \ref{weirdsemantics} should be understood in the framework of $\PZFC$, with all equivalences being strong.  The interpretations of relations $R^\M$ are not necessarily classical. For example,  it could happen that $R^\M(a)$ is both true and false,  in which case we would have  $\M \models R(x)[a]$ and $\M\not\models R(x)[a]$,  and subsequently $\M \models (R(x) \land \no R(x)) [a]$.  

\item The reader may easily verify that the defined connectives are also translated the way we would expect, namely: 
$$\M \models \lnot \varphi \iff \lnot( \M \models \varphi)$$
$$\M \models\; ! \varphi \iff \; !( \M \models \varphi)$$
$$\M \models \; ? \varphi \iff \; ?( \M \models \varphi).$$
Had we, for example, chosen $\lnot$ as the primitive connective in $\BS$ instead of $\bot$,  then Definition \ref{weirdsemantics} would have given rise to the same logic.

\item Linguistically, it is useful to distinguish {\it truth} in a Tarski model from {\it satisfaction} by a Tarski model.   We say ``$\M$ satisfies $\varphi$'' to express ``$\M \models \varphi$'' and ``$\varphi$ is true in $M$'' to express $!(\M \models \varphi)$ (which is the same as ``$\M \models \;!\varphi$'').   Satisfaction captures the entire truth value of $\varphi$ in $\M$, and we can use the notation concerning truth values from Section \ref{bzfc} to define the truth value of $\varphi$ in a Tarski model $\M$:
$$\llbracket \varphi \rrbracket_\M := \{\varnothing : \M \models \varphi\}$$ \end{enumerate}
\end{Remark}


 \begin{Thm}[$\BZFC$]\label{tarski} $\BS$ is sound and complete with respect to Tarski semantics:
 
 $$\Sigma \models \varphi  \;\; \leqqq \;\;   \Sigma \vdash_\BS \varphi.   $$
 
 %
 
 \end{Thm}

\begin{Remark}
We should note here that the bi-implication is not a strong one, i.e., soundness and completeness only talks about  {\it truth} in a model, not {\it satisfaction} by the model.  Indeed, note that $\vdash_\BS$ is a classical relation, while the Tarski-consequence relation $\models$ is not. For example, take any $\varphi$ such that some model $\N \models \varphi \land \no \varphi$, and consider the classical excluded middle $\varphi \lor \lnot \varphi$. Then for every $\M$ we have $\M \models \varphi \lor \lnot \varphi$, but one can verify that $\N \not\models \varphi \lor \lnot \varphi$. This means that ``$\varphi \lor \lnot \varphi$'' both {\it is} and {\it is not} a tautology according to Tarski semantics.



 
\end{Remark}
  We will now prove Theorem \ref{tarski}.  While this could be done  by  a  classical soundness and completeness proof  within $\BZFC$,  here we will take a short-cut by recalling   that $\BS$ is sound and complete with respect to T/F-semantics (Lemma \ref{completeness}),  arguing that this fact remains true also when adapted to $\PZFC$,  and then showing that T/F-models can be `simulated' by Tarski models, and vice versa. This `simulation' principle seems to be an interesting  phenomenon in its own right. 

First we   adapt Definition \ref{mainsemantics}:


\begin{Def}[$\PZFC$] A T/F-model is defined as in Definition \ref{mainsemantics}, with the added condition that the domain $M$, all interpretations $(R^\M)^+$ and $(R^\M)^-$ are classical subsets of $M^n$, and $=^+$ and $=^-$ are classical subsets of $M\times M$. Moreover,  it is required that for $a,b \in M$ we have $a =^+ b \; \leqqq \;\; !(a=b).$\footnote{In other words, $=^+$ is a classical relation such that $a=^+ b$ is true precisely when $a=b$ is {\it true}, and false precisely when $a=b$ is  {\it not true}.   On the  other hand $=^-$ is a classical and symmetric binary relation, but need not have anything to do with the meta-theoretic  $a\neq b$.}

\end{Def}

\begin{Lem}[$\PZFC$] \label{bs44}   $\BS$ is sound and complete with respect to T/F-semantics. \end{Lem}

\begin{proof} Since all relevant sets and relations are classical,   the original proof of Lemma \ref{completeness} can be repeated inside $\PZFC$, yielding the desired result. 
\end{proof}


\begin{Lem}[$\PZFC$] \label{one} For every Tarski model $\M$, there exists a T/F-model $\M^{\pm}$ such that for every $\varphi$:
$$\M \models \varphi \;\; \leqqq \;\;\M^\pm \models^T \varphi$$
$$\M \not\models \varphi \;\; \leqqq \;\;\M^\pm \models^F \varphi$$\end{Lem}

\begin{proof} Let $\M^{\pm}$ have the same domain as $\M$.  For every relation symbol $R$ define classical relations $(R^{\M^\pm})^+$ and $(R^{\M^\pm})^-$ (including equality) 
$$(R^{\M^\pm})^+ := \{ (a_1, \dots a_n) \in M^n \; : \;  ! R(a_1, \dots, a_n)\}$$ 
$$(R^{\M^\pm})^- := \{ (a_1, \dots a_n) \in M^n \; : \;  \lnot ? R(a_1, \dots, a_n)\}$$ 
This definition makes sure that we have
$$\M \models R[a_1 \dots a_n] \;\; \leqqq \;\; \M^\pm \models^T R[a_1 \dots a_n]$$
$$\M \not\models R[a_1 \dots a_n] \;\; \leqqq \;\; \M^\pm \models^F R[a_1 \dots a_n]$$
 An induction on the complexity of $\varphi$ then shows that the two equivalences hold for all sentences.  To exhibit an example:
$$\M \models \no \varphi \; \Leftrightarrow \; \M \not\models \varphi \; \stackrel{\rm IH}{\leqqq} \; \M^\pm \models^F \varphi \; \leqqq  \; \M^\pm \models^T \no \varphi$$
$$\M \not\models \no \varphi \; \Leftrightarrow \; \M \models \varphi \; \stackrel{\rm IH}{\leqqq} \; \M^\pm \models^T \varphi \; \leqqq \;  \M^\pm \models^F \no \varphi.$$
 We leave the details to the reader. \end{proof}

\begin{Lem}[$\BZFC$]  \label{two} For every T/F-model $\N$, there exists a Tarski model $\N^{4}$ such that for every $\varphi$:
$$\N \models^T \varphi \;\; \leqqq \;\; \N^4 \models  \varphi$$
$$\N \models^F \varphi \;\; \leqqq \;\; \N^4 \not\models  \varphi$$
\end{Lem}

\begin{proof}  First we   take care of equality, since the negative relation $=^-$ of $\N$ is not a priori related to the meta-theoretic $\neq$ for sets.  We define an equivalence relation $\equiv$ on $N$ as follows (appealing to the  Anti-Classicality Axiom):

$$a \equiv b \; \leqqq \; a = b$$
$$a \not\equiv b \; \leqqq \; a =^- b$$
 It is not hard to verify that $\equiv$ satisfies Definition \ref{eqrel}, so if we consider equivalence classes $[a] := \{b \in N : a \equiv b\}$ then by Lemma \ref{eqrel2} we know that $$[a] = [b] \leqqq a = b$$ $$[a] \neq [b] \leqqq  a =^- b.$$

\p Let the domain of $\N^4$ be $  \{ [a] : a \in N\}$ and notice that we have
$$\N \models^T (t=s)[a,b] \; \leqqq \; a=b \; \leqqq \; [a] = [b] \; \leqqq \; \N^4 \models (t=s)[ [a] ,[b]]$$
$$\N \models^F (t=s)[a,b] \; \leqqq \; a=^- b \; \leqqq \; [a] \neq [b] \; \leqqq \; \N^4 \not\models (t=s)[[a] ,[b]]$$

 \p For every relation symbol $R$,   define the interpretation $R^{\N^4}$ by:

$$\;\; R^{\N^4}( [a_1]\dots [a_n]) \; \leqqq \; (R^\N)^+(a_1 \dots a_n)$$
$$\no R^{\N^4}( [a_1]\dots [a_n]) \; \leqqq \; (R^\N)^-(a_1 \dots a_n)$$
which is again possible by appealing to the Anti-Classicality Axiom.  Also note that this is well-defined by the usual arguments.   

  \p Finally, we leave it to the reader to verify,  by induction on the complexity of $\varphi$, that for all $a_1, \dots, a_n \in N$ we have

$$\N \models^T \varphi[a_1 \dots a_n] \;\; \leqqq \;\; \N^4 \models \varphi[[a_1] \dots [a_n]]$$
$$\N \models^F \varphi[a_1 \dots a_n] \;\; \leqqq \;\; \N^4 \not\models \varphi[ [a_1] \dots [a_n]]$$

\p In particular this is true for all sentences $\varphi$, completing the proof.\end{proof}


\begin{proof}[Proof of Theorem \ref{tarski} $(\BZFC)$] By Lemma  \ref{bs44}, $\BS$ is sound and complete with respect to T/F-semantics. But by Lemmas  \ref{one} and \ref{two}, T/F-models can be replaced by Tarski-models and vice versa. Thus T/F-semantics are equivalent to Tarski-semantics, which completes the proof.
\end{proof}


Notice that while most of the theory in this section only required $\PZFC$, Lemma  \ref{two}  hinges on the Anti-Classicality Axiom.  Indeed,  we can now analyse the situation in more detail.

Recall from  Remark \ref{moregeneral}    that if one is interested in a set theory which has only inconsistent but not incomplete sets, or only incomplete but not inconsistent sets, one can consider these two fragments of the Anti-Classicality Axiom:

\begin{itemize} 
\item  $\boldsymbol{\ACLA_{\rm cons}:} \boldsymbol{  \;\; \forall x (x^! \subseteq x^?) \land \exists x (x^? \not\subseteq x^!)}$

(all sets are consistent, but there is an incomplete set).
\item $\boldsymbol{\ACLA_{\rm comp} :  \;\; \forall x (x^? \subseteq x^!) \land \exists x (x^! \not\subseteq x^?)}$  

(all sets are complete, but there is an inconsistent set).

\end{itemize}
Now we consider some fragments of  $\BS$, referring to them by the name under which they (or a very similar version) have previously appeared in the literature. 


\begin{Def}[PZFC] $\;$

 \begin{enumerate}
 
\item  K3$^\to$: this  logic is obtained if in Definition \ref{sem1} we require   $(R^\M)^+ \cap (R^\M)^- = \varnothing$  for all relation symbols $R$ (including equality).  The propositional version  appeared  in  \cite{K3arrow}.
  \item  LFI1: this  logic is obtained if in Definition \ref{sem1} we require  $(R^\M)^+ \cup (R^\M)^- = M$ (the whole domain) for all relation symbols $R$ (including equality).  The propositional version  appeared in \cite{LFI1} .
  
  \item FDE: the ``$\to$''- and ``$\bot$''-free fragment of $\BS$,   which appeared  in \cite{Asenjo-LP}.
  \item K3: the  ``$\to$''- and ``$\bot$''-free fragment of K3$^\to$.  The propositional version  is the widely known three-valued  logic with indeterminate truth values of Kleene \cite{KleeneK3}.
  \item LP: the  ``$\to$''- and ``$\bot$''-free fragment of LFI1. The propositional version is the well-known {\it logic of paradox} of Priest \cite{Priest-LP}.


\end{enumerate} 
\end{Def}

K3 and K3$^\to$ are three-valued logics designed to deal only with incompleteness but not inconsistency,  while LP and LFI1 are also three-valued and deal inconsistency but not incompleteness.  FDE and $\BS$ are four-valued and take incompleteness as well as inconsistency into account.  FOL refers to classical logic.  







\begin{Thm}[PZFC] \label{tt1} $\;$

\begin{enumerate}
\item  $  \ACLA  \; \leqqq \; (  \Sigma \models \varphi  \;\; \leqqq \;\;   \Sigma \vdash_{\: \BS} \varphi).  $
\item $\ACLA_{\rm cons}  \; \leqqq \; (  \Sigma \models \varphi  \;\; \leqqq \;\;   \Sigma \vdash_{ \:{\rm K3}^\to} \varphi).  $
\item $\ACLA_{\rm comp}  \; \leqqq \; (  \Sigma \models \varphi  \;\; \leqqq \;\;   \Sigma \vdash_{\: \rm LFI1} \varphi).  $
\item $\forall x (x^! = x^?)  \; \leqqq \; (  \Sigma \models \varphi  \;\; \leqqq \;\;   \Sigma \vdash_{ \: \rm FOL}  \varphi).  $
\item Exactly one of the above   holds.
\end{enumerate}
\end{Thm}
 If we restrict attention to formulas not containing implications or $\bot$,  we obtain the same result for  more familiar systems.

\begin{Thm}[PZFC] \label{tt2} $\;$ Suppose $\Sigma$ and $\varphi$ do not contain ``$\to$'' or ``$\bot$''. Then:

\begin{enumerate}
\item  $  \ACLA  \; \leqqq \; (  \Sigma \models \varphi  \;\; \leqqq \;\;   \Sigma \vdash_{\: \rm FDE} \varphi).  $
\item $\ACLA_{\rm cons}  \; \leqqq \; (  \Sigma \models \varphi  \;\; \leqqq \;\;   \Sigma \vdash_{ \:{\rm K3}} \varphi).  $
\item $\ACLA_{\rm comp}  \; \leqqq \; (  \Sigma \models \varphi  \;\; \leqqq \;\;   \Sigma \vdash_{\: \rm LP} \varphi).  $
\item $\forall x (x^! = x^?)  \; \leqqq \; (  \Sigma \models \varphi  \;\; \leqqq \;\;   \Sigma \vdash_{ \: \rm FOL}  \varphi).  $
\item Exactly one of the above  holds.
\end{enumerate}
\end{Thm}  The proofs of  these theorems  are variations of  Theorem \ref{tarski}, and are left to the reader.

 \section{  Future Work} \label{conclusion}

 We would like to conclude by reflecting on a number of possible future research directions.  

 \begin{enumerate}
 
 \item  \textbf{Algebraic approach.} The logic $\BS$ and set theory in $\BS$ can be approached using algebraic semantics, specifically via   so-called {\it twist algebras}. This can even be used to study an analogue of Boolean-valued models (closely related to forcing over classical models of $\ZFC$) in the $\PZFC$- or $\BZFC$-context.  A significant part of this has  been done in \cite[Chapter 3 and Chapter 8]{HrafnThesis} but there are more things that can be studied.  
 
 \item \textbf{Constructive $\BZFC$.} It is also possible to consider an intuitionistic version of $\BS$, defined syntactically by changing axioms 1--14 in Section \ref{proof} to an intuitionistic version,  or semantically by considering Kripke frames with T/F-models at the nodes.  In this logic, known as N4 and appearing in \cite{NelsonN4},    $\varphi \lor \no \varphi$ and $\varphi \lor \lnot \varphi$ both need not be true, and $\varphi \land \no \varphi$ does not lead to a strong contradiction, but $\varphi \land \lnot \varphi$ does. Moreover, $\no \no \varphi$ is strongly equivalent to $\varphi$ while $\lnot \lnot \varphi$ is not. 
 
 What type of set theory do we obtain in such a logic if we combine the ideas of this paper with constructive systems like $\IZF$ or $\CZF$? 

\item \textbf{Computability theory in $\BZFC$.}  There seems to be a natural application of $\BZFC$ regarding computability theory.  Recall that  a Turing machine $M$   \emph{computes} a set $A \subseteq \IN$ if for all $n$: \begin{itemize}
\item $n \in A  \; \leqqq \; M$ halts on input $n$ and outputs some non-$0$ value.
\item $n \notin A  \; \leqqq \; M$ halts on input $n$ and outputs $0$.
\end{itemize}
Moreover, let us say that $M$ \emph{recognizes} a set $A \subseteq \IN$ if for all $n$:
\begin{itemize}
\item $n \in A  \; \leqqq \; M$ halts on input $n$  and outputs a non-$0$ value.  
\end{itemize}
Classically, every Turing machine recognizes a set, but not every Turing machine computes a set (since it may not halt on a given input).  Thus, one cannot use sets $A \subseteq \IN$ to represent decisions of a machine.

    
    
  \p  
Let us revisit the situation in BZFC (or even $\text{PZFC}+\text{ACLA}_\text{cons}$). Using the same definition as above,  it follows from Theorem \ref{aclaacla} that every Turing machine computes a set.  If a given machine does not halt on input $n$, it means only that it  computes an incomplete set $A$, namely a set for which $\lnot (n \in A \lor n \notin A)$. In particular, the decision process of a machine can be completely described by subsets of $\IN$. 


 \item \textbf{Cardinal arithmetic in $\BZFC$.}  It is clear that  we  need a new notion of \emph{cardinality} and \emph{cardinal numbers} to describe the size of  non-classical sets.  Such a notion should  capture the size of the entire structure of a non-classical sets, i.e.,   the size of the $!$-extension, $?$-extension, and all of the combinations,  in one go.  Some preliminary work on this has already been done, and we expect that this will  lead to a rich theory of cardinals and cardinal arithmetic.

\end{enumerate}

 \vbox{
 \s{Acknowledgments}
 
\p  We would like to thank  Benno van den Berg, Luca Incurvati, Benedikt L\"owe, Hitoshi Omori, Andrew Tedder, Giorgio Venturi and Heinrich Wansing for insightful conversations, helpful suggestions and ideas.  We also want to thank the anonymous referee whose thorough reading of an earlier version of this paper has  contributed significantly to an improved presentation.
}

\bibliographystyle{plain}
\bibliography{../../10_Bibliography/Khomskii_Master_Bibliography}{}

\begin{thebibliography}{10}

\bibitem{AFA}
Peter Aczel.
\newblock {\em Non-well-founded sets}, volume~14 of {\em CSLI Lecture Notes}.
\newblock Stanford University, Center for the Study of Language and
  Information, Stanford, CA, 1988.
\newblock With a foreword by Jon Barwise [K. Jon Barwise].

\bibitem{NelsonN4}
Ahmad Almukdad and David Nelson.
\newblock Constructible falsity and inexact predicates.
\newblock {\em The Journal of Symbolic Logic}, 49(1):231–233, 1984.

\bibitem{Asenjo-LP}
F.~G. Asenjo.
\newblock A calculus of antinomies.
\newblock {\em Notre Dame J. Formal Logic}, 7:103--105, 1966.

\bibitem{Avron}
Arnon Avron.
\newblock Natural {$3$}-valued logics---characterization and proof theory.
\newblock {\em J. Symbolic Logic}, 56(1):276--294, 1991.

\bibitem{Andrew-Bacon}
Andrew Bacon.
\newblock Non-classical metatheory for non-classical logics.
\newblock {\em J. Philos. Logic}, 42(2):335--355, 2013.

\bibitem{ExtraBS4}
Diderik Batens, Kristof De~Clercq, and Natasha Kurtonina.
\newblock Embedding and interpolation for some paralogics. {T}he propositional
  case.
\newblock {\em Rep. Math. Logic}, (33):29--44, 1999.

\bibitem{Belnap1}
Nuel~D. Belnap.
\newblock How a computer should think.
\newblock In {\em New essays on {B}elnap-{D}unn logic}, volume 418 of {\em
  Synth. Libr.}, pages 35--53. Springer, Cham, [2019] \copyright 2019.

\bibitem{Belnap2}
Nuel~D. Belnap.
\newblock A useful four-valued logic.
\newblock In {\em New essays on {B}elnap-{D}unn logic}, volume 418 of {\em
  Synth. Libr.}, pages 55--76. Springer, Cham, [2019] \copyright 2019.

\bibitem{Carnielli}
Walter Carnielli and Marcelo~E. Coniglio.
\newblock Paraconsistent set theory by predicating on consistency.
\newblock {\em J. Logic Comput.}, 26(1):97--116, 2016.

\bibitem{LFI1}
Walter Carnielli, Marcelo~E. Coniglio, and Jo{\~a}o Marcos.
\newblock {\em Logics of Formal Inconsistency}, pages 1--93.
\newblock Springer Netherlands, Dordrecht, 2007.

\bibitem{Dunn}
J.~Michael Dunn.
\newblock Intuitive semantics for first-degree entailments and `coupled trees'.
\newblock {\em Philos. Studies}, 29(3):149--168, 1976.

\bibitem{Geach}
P.~T. Geach.
\newblock {On Insolubilia}.
\newblock {\em Analysis}, 15(3):71--72, 01 1955.

\bibitem{K3arrow}
Allen~P. Hazen and Francis~Jeffry Pelletier.
\newblock K3, {L}3, {LP}, {RM}3, {A}3, {FDE}, {M}: how to make many-valued
  logics work for you.
\newblock In {\em New essays on {B}elnap-{D}unn logic}, volume 418 of {\em
  Synth. Libr.}, pages 155--190. Springer, Cham, [2019] \copyright 2019.

\bibitem{Hinnion}
Roland Hinnion.
\newblock About the coexistence of ``classical sets'' with ``non-classical''
  ones: a survey.
\newblock Number 11-12, pages 79--90. 2003.
\newblock Flemish-Polish Workshops II--IV and varia.

\bibitem{KleeneK3}
S.~C. Kleene.
\newblock On notation for ordinal numbers.
\newblock {\em Journal of Symbolic Logic}, 3(4):150?155, 1938.

\bibitem{KunenInconsistency}
Kenneth Kunen.
\newblock Elementary embeddings and infinitary combinatorics.
\newblock {\em J. Symbolic Logic}, 36:407--413, 1971.

\bibitem{Libert}
Thierry Libert.
\newblock Models for a paraconsistent set theory.
\newblock {\em J. Appl. Log.}, 3(1):15--41, 2005.

\bibitem{HrafnThesis}
Hrafn~Valt{\'y}r Oddsson.
\newblock Paradefinite {Z}ermelo-{F}raenkel set theory: A theory of
  inconsistent and incomplete sets.
\newblock Master's thesis, Universiteit van Amsterdam, 2021.
\newblock ILLC Publication MoL-2021-27.

\bibitem{Observations}
Hitoshi Omori and Toshiharu Waragai.
\newblock Some observations on the systems {LFI}1 and {LFI}1.
\newblock In {\em 22nd International Workshop on Database and Expert Systems
  Applications}, pages 320--324, 2011.

\bibitem{RobertP}
Robert Passmann.
\newblock Should pluralists be pluralists about pluralism?
\newblock {\em Synthese}, 199(5-6):12663--12682, 2021.

\bibitem{Priest-LP}
Graham Priest.
\newblock The logic of paradox.
\newblock {\em J. Philos. Logic}, 8(2):219--241, 1979.

\bibitem{Strongimplication}
Helena Rasiowa.
\newblock {\em An algebraic approach to non-classical logics}.
\newblock Studies in Logic and the Foundations of Mathematics, Vol. 78.
  North-Holland Publishing Co., Amsterdam-London; American Elsevier Publishing
  Co., Inc., New York, 1974.

\bibitem{NoteNaive}
Greg Restall.
\newblock A note on naive set theory in {${\rm LP}$}.
\newblock {\em Notre Dame J. Formal Logic}, 33(3):422--432, 1992.

\bibitem{BS4complete}
Katsuhiko Sano and Hitoshi Omori.
\newblock An expansion of first-order {B}elnap-{D}unn logic.
\newblock {\em Log. J. IGPL}, 22(3):458--481, 2014.

\bibitem{waterplants}
Andrea Sereni and Maria Fogliani.
\newblock How to water a thousand flowers. on the logic of logical pluralism.
\newblock {\em Inquiry}, 63:1--24, 09 2017.

\bibitem{Shapiro-Varieties}
Stewart Shapiro.
\newblock {\em Varieties of Logic}.
\newblock Oxford University Press, 2014.

\bibitem{Weber}
Zach Weber.
\newblock {\em Paradoxes and Inconsistent Mathematics}.
\newblock Cambridge University Press, 2021.

\end{thebibliography}

\end{document}